\documentclass[a4paper, 11pt]{article}
\usepackage{dsfont}

\usepackage{calrsfs}

\usepackage{amsmath, amsthm, amssymb, amsfonts}

\usepackage[shortlabels]{enumitem}
\usepackage[square]{natbib}
\usepackage{float}

\usepackage[usenames]{color}

\usepackage{ulem}

\usepackage[colorlinks=true,breaklinks=true,bookmarks=true,urlcolor=blue,
     citecolor=blue,linkcolor=blue,bookmarksopen=false,draft=false]{hyperref}
     
\newcommand{\ignore}[1]{}

\newtheorem{theorem}{Theorem}

\newtheorem{corollary}[theorem]{Corollary}

\newtheorem*{theorem_nonum}{Theorem}

\newtheorem{lemma}[theorem]{Lemma}

\newtheorem{proposition}[theorem]{Proposition}

\newtheorem{definition}[theorem]{Definition}

\newtheorem{remark}[theorem]{Remark}

\title{Price dynamics on a risk-averse market with asymmetric information}

\author{Bernard De Meyer \thanks{PSE-Universit\'e Paris 1, 106-112 Boulevard de l'H\^opital, 75647 Paris Cedex 13, France, \texttt{demeyer@univ-paris1.fr}}
\and
Ga\"etan Fournier \thanks{IAST - Toulouse 1 Capitole, 21 All\'ee de Brienne, 31000 Toulouse France,  ~ \texttt{fournier.gtn@gmail.com} Support through the ANR Labex IAST is gratefully acknowledged.} 
}

\begin{document}

\maketitle

\begin{abstract}
A market with asymmetric information can be viewed as a repeated exchange game between  the informed sector and the uninformed one. In a market with risk-neutral agents, \citet{de2010} proves that the price process should be a particular kind of Brownian martingale called CMMV. This type of dynamics is due to the strategic use of their private information by the informed agents. In the current paper, we consider the more realistic case where agents on the market are risk-averse. This case is much more complex to analyze as it leads to a non-zero-sum game. Our main result is that the price process is still a CMMV under a martingale equivalent measure. This paper provides thus a theoretical justification for the use of the CMMV class of dynamics in financial analysis. This class contains as a particular case the Black and Scholes dynamics.
\end{abstract}



\section{Introduction}\label{intro}
Information asymmetries are omnipresent in financial markets. We do not mean here insider trading which is illegal but agents on the market are de facto asymmetrically informed: institutionals have clearly a better access to information than private investors. They have access to more information, quicker, and they are better skilled to analyze it. Typically, they have more information on the companies whom shares they are trading and they have entire services analyzing the economic conjuncture. Aside these information about the economic health of the underlying firms, they also have access to an other kind of private information that is relevant to forecast the price evolution: they often serve as intermediary between their clients and the market. When receiving an important order from a client, they clearly receive a private information that will affect the short term price evolution. They will typically use this information in an optimal way to get the best execution price. This is the focus of the literature on optimal execution, as introduced in \cite{almgren2001optimal}. \\

In these situations everybody is aware that informational asymmetries exist and knows who are the informed agents. Informed agents' actions on the market are therefore analyzed by the uninformed agents in order to deduce the informative content behind these actions. As suggested in previous papers (\citet{de2003} and \citet{de2010}), this phenomenon could partially explain the  kind of price dynamics observed on the market.\\

These papers model the market with a single risky asset which is exchanged for counterpart a num\'{e}raire. For simplicity they consider a short period of time  just after a single asymmetric information shock. Because the time period considered is short, it can be assumed that there is no consumption and the aim of the agent is to maximize the expected value of their final portfolio. The informational shock is materialized by a initial private random message $m$ received by the informed player. This message will influence the final price $L$ of the risky asset which can thus be viewed as a deterministic function $L=L(m)$. $L(m)$ contains in fact all the relevant information carried by $m$. Therefore the structure of the initial information asymmetry can be simplified. 
Since $m$ is random, so is $L(m)$. Denoting $\mu$ the probability distribution of $L(m)$, the initial informational shock is now modeled by a initial lottery selecting $L$ with probability $\mu$ once for all. Player 1 is informed of $L$ while Player 2 just knows $\mu$. After this initial information shock, players are exchanging assets during $n$ consecutive trading rounds, using a trading mechanism. This is thus a repeated exchange game of incomplete information \`{a} la \cite{au1995}.\\

\citet{de2003} analyzes a game between two risk-neutral market makers with asymmetric information. They focus on a very particular trading mechanism and proved that as market markers play more and more frequently ($n$ going to infinity), the equilibrium price process converges to a continuous martingale involving a Brownian motion in its description. This result gives thus an endogenous justification for the appearance of the Brownian term in the price dynamics: it is seen as an aggregation of the random day after day noises introduced  by the informed agents on their moves to avoid too fast revelation of their private information.\\

This idea was generalized in \citet{de2010} that argues that a market with incomplete information  can be modeled by a two player game between the informed sector and the uninformed sector. In first approximation these sectors are considered in \citet{de2010} as individually rational and risk-neutral agents. In such a description the uninformed sector is typically made of an aggregation of a large number of agents. It is then difficult to describe precisely the set of all possible actions of all those agents (which would be a complete action profile, one for each individual agent). This is the reason for modeling the market by an abstract trading mechanism. Such a mechanism simply maps a pair of actions to the resulting share transfer between the two sectors. The uninformed sector is then represented by a "representative agent" that selects rationally the action in that action space. Five conditions are imposed on the trading mechanism to model real world markets. When those conditions are satisfied, the trading mechanism is called natural.

In games with natural trading mechanisms, it appears that the equilibrium price processes converge, as the trading frequency increases. The limit process belongs to a very particular class of Brownian martingale referred to as the class of "Continuous Martingale of Maximal Variation" (CMMV, see below definition \ref{def:CMMV}).

In fact, the particular mechanism analyzed in \citep{de2003} is a natural mechanism and the dynamics observed in that paper is a particular CMMV. Let us emphasize that the asymptotic behavior of the prices is completely independent of the natural trading mechanism used to model the market. We refer to that result as the universality of the CMMV class.\footnote{This universality of the CMMV class is a strong and surprising result. It is of the same vein as the universality of the normal law in the central limit theorem. Actually, the convergence to a CMMV is proved in \citep{de2010} as a consequence of the central limit theorem.}~~\\

This universality of the CMMV class is still reinforced by the result of \citet{fabien_these}. This class is robust to the introduction of classical derivative assets: instead of considering just one risky asset, he considers a multi-asset model with one underlying asset and a family of monotonic derivatives. In that framework he shows that the limit of the price process of each asset is a CMMV.
~~\\

In the present paper, we extent this universality result, showing that the CMMV class also appears in a model with risk-averse agents. This result suggests  that the CMMV class should be used in finance to model asset price evolutions. Notice that Black and Scholes dynamics is a particular CMMV.

\section{The content of this paper}
The current paper analyzes the consequences of introducing risk aversion in the model. As mentioned above, the uninformed sector is made of a large number of individual agents that typically display risk-averse behavior. It is therefore natural to assume that the representative agent, called player 2, will be risk-averse. On the other hand, the informed agent (player 1) typically represents a big institutional investor and it is natural to model it as a risk-neutral agent.\\

This risk aversion is modeled with the introduction of a concave utility function in player 2's payoff: he maximizes the expected utility of the final value of its portfolio. Due to this utility function, we are not in front of a zero-sum game anymore as it was the case in the previous mentioned papers. This makes the analysis more involved, the notions of value and optimal strategies are here to be replaced by the notion of Nash equilibrium. \\

When dealing with very general trading mechanism with abstract action spaces, the notion of price is not obvious. In the setting of risk neutrality, as in \cite{de2010}, the price $L_t$ at time $t$ of the risky asset is defined  as the conditional expectation of its final value $L$ given the public information at that time. This makes sense in the risk-neutral case since this conditional expectation is precisely the price at which the uniformed agent would agree to trade this asset. But this doesn't make sense anymore in a risk-averse setting. We chose to bypass this issue by considering a particular exchange mechanism that naturally involves prices.


There are so many uninformed agents on the market that aside from its informational content, the action of player 1 will be quite marginal and the market without player 1 can be considered as a device that produces a price at which player 1 can buy or sell a unit of the asset. This amounts to view player 2 as a market maker. 

We consider therefore a very simple mechanism where the uninformed sector is a market maker that chooses at each period $q \in \{1\dots,n\}$ a price $p_q$ for one share of the risky asset, and player 1 will have to decide whether he wants to buy ($u_q=1$) or to sell ($u_q=-1$) at this price. Note that the price $p_q$ is thus the number of shares of the num\'{e}raire given in exchange for $1$ share of the risky asset. Since we suppose that the num\'{e}raire has a liquidation value equal to $1$, the price $p_q$ as thus to be interpreted as an actualized price of the risky asset. Both players try to maximize their utility for the liquidation value of their final portfolio.\\

We first prove the existence of a Nash equilibrium for a game with fixed length $n$. We then analyze equilibrium strategies. Under those strategies we analyze the law of the price process $(p_1,\dots,p_n)$ posted by player 2. Our next result is that this price dynamics is consistent with the classical financial theory of no-arbitrage: the so-called "fundamental theorem of finance" by \citet{ha1981} claims that if there is no arbitrage on the market, there exists an equivalent probability measure\footnote{Two probabilities on a probability space are equivalent if they have the same events of probability $0$.}, under which the actualized price process would be a martingale. In our model with risk aversion, the (actualized) equilibrium price process is not, in general, a martingale. However, as in \citet{ha1981}, there exists an equivalent probability measure under which this process become a martingale. 

This result is quite surprising in a context where Player 1 can only buy or sell a limited amount of assets. Indeed, the no-arbitrage theory assumes that if there would exist a trading strategy leading to a positive final value at no initial cost, there would be such a demand for this portfolio that the prices on the market would be affected and this arbitrage would disappear. This argument fails in our model because in any case Player 1 can only buy a limited amount at each period.\\

We then analyze the asymptotics of the price dynamics as the trading frequency increases ($n$ goes to infinity). The limit price process is a process in continuous time (where $t \in [0,1]$ is the proportion of elapsed trading periods). Our result is that under equivalent probability measure mentioned above, the limit price process is a CMMV.\\

Again, this result reinforces the universality of the CMMV class. More precisely it suggests that under the martingale equivalent probability measure, the price processes of the assets should be a CMMV. Note in particular that this class contains the most used dynamics in finance, which is one used in \cite{black1973pricing}. But there are plenty other dynamics in this class, that could be used to develop pricing and hedging models.\\

We now give a precise definition of CMMV:

\begin{definition}\label{def:CMMV}
A continuous martingale of maximal variation\footnote{The terminology "CMMV" was introduced in \citet{de2010} due to the following result. The $n$-variation of a martingale $(X_t)_{t \in [0,1]}$ is $V_X^n = \sum_{q=0}^{n-1} \| X_{\frac{q+1}{n}} - X_{\frac{q}{n}} \|_{L^1}$. Consider the problem $\mathcal{M}_n$ of maximizing the $n$-variation $V^n_X$ on the class of martingales $X$ with final distribution $\mu$ ($X_1 \sim \mu$). It is proved in \citet{de2010} that the martingales that solve $\mathcal{M}_n$, (i.e. martingales of maximal variation) converge in finite distributions, as $n$ goes to infinity, to a process that satisfies the above Definition \ref{def:CMMV}.} (CMMV) is a stochastic process $\Pi$ in continuous time $t \in [0,1]$,  which is a martingale that satisfies for all $t$: 

\[
\Pi_t=f(B_t,t)
\]
where B is a standard Brownian motion and $f: \mathbb{R}\times [0,1] \rightarrow \mathbb{R}$ is a function which is increasing in its first variable.
\end{definition}

Let us emphasize that they are two conditions in the above definition. One is that $f(x,t)$ is increasing with $x$, and the other is that $f(B_t,t)$ is a martingale, which implies strong restriction on $f$. It follows in particular from It\^{o}'s formula that  $f$ must satisfy the time reversed heat equation: 
\[  \frac{\partial f}{\partial t}+ \frac{1}{2} \frac{\partial^2 f }{\partial x^2 }=0.
\]

\section{Description of the model}~~

Let $\mu$ be a probability distribution on $\mathbb{R}$. The game $G_n(\mu)$ we are considering is the $n$-times repeated games that proceeds as follows: at stage $0$, nature selects once for all $L$ at random with probability distribution $\mu$. Player 1 observes $L$, not player 2, and both players know $\mu$. This initial information period is followed by $n$ trading periods. At each period $q \in \{1,\dots,n\}$, player 1 decides to buy ($u_q=1$) or to sell ($u_q=-1$) one unit of the risky asset. $u_q \in \{+1,-1\}$ is thus the action of player 1. Simultaneously, player 2 selects the price $p_q \in \mathbb{R}$ of the transaction at stage $q$.

\begin{remark}
Choices are thus considered to be simultaneous: in our model, player 1 does not observe player 2's action before deciding whether to sell or buy. This can be surprising at first glance. Indeed, one usually assumes that the trader will buy or sell after observing the current market maker's price. In fact, we argue that this sequential model where player 1 reacts to the price posted by player 2 is equivalent to our model. Indeed, we prove in section \ref{section_reduced_eq} that, due to Jensen's inequality, the equilibrium strategy of player 2 in the simultaneous game is a pure strategy. Player 2's move $p_q$ is thus completely forecastable for player 1 at period $q$. Player 1 would get no benefit from observing $p_q$ before selecting $u_q$.
Therefore the equilibria in the simultaneous game are also equilibria in the sequential game.

That the game can be seen as a sequential game does not make mixed strategies useless. Indeed, in a sequential game with full information players select at each stage the action that will maximize their continuation payoff and this can be done with a pure strategy. However, this game is of incomplete information and mixing is the keystone for player 1 to avoid too fast revelation of his private information.
\end{remark}

~~\\

We denote $h_{q}$ the history of plays until round $q$, i.e. $h_{q}=(u_1, p_1, \dots, u_q, p_q)$ 
and $\mathbb{H}_{q}$ the set of all possible histories until round $q$. At the end of stage $q$, $u_q$ and $p_q$ are publicly revealed. Then both players know and remember all the past actions taken by both of them. Since the game has perfect recall we can apply Kuhn's theorem and there is no loss of generality to assume that players use behavioral strategies. 

A behavioral strategy for player 1 in this game is a sequence $\sigma = (\sigma_1 , \dots , \sigma_n )$ with $\sigma_q : (h_{q-1} , L)  \rightarrow \sigma_q(h_{q-1} , L) \in \Delta( \{ -1, +1 \} )$, with the usual convention that $\Delta(S)$ is the set of probabilities on a finite set $S$. A behavioral strategy for player 2 is a sequence $\tau =(\tau_1 , \dots , \tau_n )$ with $\tau_q : h_{q-1}  \rightarrow \tau_q(h_{q-1}) \in \Delta( \mathbb{R})$, where $\Delta(\mathbb{R})$ denote the set of Borel probabilities on $\mathbb{R}$. A triple $(\mu,\sigma,\tau)$ induces a unique probability distribution  for $(L,h_n)$. When $X$ is a random variable, we denote $E_{\mu, \sigma, \tau} [X]$ its expectation with respect to this probability.

In this paper, player 1 is risk-neutral. His payoff in $G_n(\mu)$ is then the expected value of his final portfolio, up to the normalization factor $\frac{1}{\sqrt{n}}$ that will be explained in the fore-coming remark:
\begin{equation}\label{formula_payoff_p1}
g_1(\sigma, \tau)= E_{\mu, \sigma, \tau} \left[\frac{1}{\sqrt{n}}\sum_{q=1}^n u_q(L-p_q)\right]
\end{equation} 
The particularity of the current paper is that we consider a risk-averse player 2. The payoff he aims to minimize (we keep the formalism of the zero-sum games where player 2 is a minimizer) is then:
\begin{equation}\label{formula_payoff_p2} 
g_2(\sigma, \tau )= E_{\mu, \sigma, \tau} \left[H \left( \frac{1}{\sqrt{n}}\sum_{q=1}^n u_q(L-p_q )\right)\right]
\end{equation}
where H is a risk aversion function (convex and increasing).
\begin{remark}
The  normalization factor $\frac{1}{\sqrt{n}}$ introduced in the payoff functions requires some explanations. In this paper, we ultimately aim to analyze  the price dynamics on a market  in continuous time. We approach this continuous time market by a discrete time one, and the market is modeled  as a repeated exchange game between a market maker and an informed player, with a large number $n$ of repetitions.
At each period, there is a maximal quantity $\alpha_n$ of asset 1 that can be bought or sold at the price posted by the market maker. 
This maximal quantity $\alpha_n$ measures some how the liquidity of the market: it represents  the quantity of asset $R$ that can be exchanged at once at the market price without affecting the price. 
The existence of a maximal quantity acts as  a protection for the market maker against an insider trading. 
Would player 2 be a single fully rational player, he would prefer avoid trading with a more informed player 1, due to a kind of "No trade" theorem (see \citet{milgrom}). 
Because player 2 is an aggregation of different players, some of whose are in fact trading for liquidity reasons, he is unable to completely avoid trading and this is reflected here by setting $\alpha_n>0$.

We next argue that $\alpha_n$ should be proportional to $\frac{1}{\sqrt{n}}$.  Indeed, this appears clearly in the model where player 2 is risk-neutral. The function $H$ is then linear and the game can be considered as zero-sum. 
In this setting the behavior of the players is independent of $\alpha_n$ which is just a normalisation factor of utilities. In fact  this particular game corresponds to a natural exchange mechanism. It can be easily shown that this mechanism satisfies the five hypothesis mentioned above, and we know from \citet{de2010} that the value $V_n$ of the game is such that $\frac{V_n}{\sqrt{n}}$ converges to a finite quantity. This result points out that $\alpha_n$ should be taken proportional $\frac{1}{\sqrt{n}}$ in order to stay with bounded payoffs as $n$ increases. The exchanged quantities will then remain bounded and in the limit converge to the quantities exchanged in the continuous time model.
This normalization has no effect on the players in the risk-neutral case, but it has to be introduced  in the risk-averse case:  the lotteries player 2 is facing are of the same magnitude for all $n$.

\end{remark}

Throughout this paper, we will make the following regularity assumptions on $\mu$ and $H$:\\

\textbf{A1}: $\mu$ is a probability measure on $[0,1]$ absolutely continuous with respect to the Lebesgue measure. Its density function $f_{\mu}$ is strictly positive and $C^1$.\\

\textbf{A2}: $H$ is a strictly positive, strictly convex and $C^2$ function and $H'$ is Lipschitz-continuous: there exists strictly positive $\epsilon$ and $K$ such that for all $x \in \mathbb{R}: \epsilon < H'(x) < K$.

~~\\

Observe that in \textbf{A1} we assume that $L$ takes only values in the $[0,1]$ interval. We could obviously change this assumption to any compact interval by just a renormalization. \\

The game analyzed in this paper is in fact quite simple and much simpler than the general games with abstract strategy spaces considered in \citep{de2010}. However our proofs are very long and we have to apologize for the technicality of the following pages. The difficulties in the proofs come from different reasons.

The first difficulty is that strategies of the players in the game with $n$ rounds of trading are defined on different spaces. For example we will identify Player 1's strategy set with the set of probabilities on $\{-1,+1 \}^n \times \mathbb{R}$. Is it therefore quiet difficult to speak of the convergence of those strategies. We bypass this difficulty using embedding methods \`{a} la Skorokod, which are per se technical.

Next, equilibria in the finite game are defined implicitly as a fixed point, and there are no closed form formula to describe them. Proving the convergence of these equilibria is not straightforward.

Finally, we prove that the convergence is equivalent to the uniqueness of solution to a differential system with very strong nonlinearities. Those nonlinearity involve in particular the arbitrary risk aversion function $H$ and the arbitrary density function $f_{\mu}$. This differential system is not covered by the classical literature on differential equations.

Due to the length of these proofs, we decided to make a first synthetic overview of the argument in the next section.\\

\section{Results and structure of the paper}~~

In the first part of the paper (sections \ref{section_reduced_eq}, \ref{section_conditions}, \ref{section_existence}), we analyse the game $G_n(\mu)$ for a fixed number $n$ of stages. We first prove in section \ref{section_reduced_eq} that some equilibria of $G_n(\mu)$ can be found among the equilibria of the simpler game $\overline{G}_n(\mu)$ where the informed player 1 does not observe the actions of player 2, and player 2 is not allowed to randomize his moves. We then focus on the reduced game $\overline{G}_n(\mu)$. 

We show in section \ref{section_conditions} that this game can be completely reformulated: a strategy of player 1 can be identified with a probability $\Pi_n$ on the pair $(\omega,L)$ where $\omega=(u_1,\dots,u_n)$. We prove in subsection \ref{subsection:charac} that at an equilibrium, $\frac{\sum u_k p_k}{\sqrt{n}}$ should be equal to $\Psi_n(S_n(\omega))$ where $\Psi_n$ is a convex function and $S_n(\omega):=\frac{1}{\sqrt{n}}\sum_{k=1}^{n} u_k$. The function $\Psi_n$ can be used to parameterized player 2's strategies. Indeed, player 2's strategy $(p_1,\dots,p_n)$ can be recovered from $\Psi_n$. A pair of strategies can therefore be described by a pair $(\Pi_n,\Psi_n)$.

We further show in the same subsection that, in order to be an equilibrium in $\overline{G}_n(\mu)$, $(\Pi_n,\Psi_n)$ must satisfy the following conditions \textbf{(C1)} to \textbf{(C4)}, where  $\overline{\Pi}_n$ denotes the marginal of $\Pi_n$ on $(L,S_n)$,  where $\lambda_n$ denotes the uniform probability on $\omega$ and where $\overline{\lambda}_n$ is the law of $S_n(\omega)$ when $\omega\sim\lambda_n$.\\

\textbf{(C1)}: $E_{\overline{\lambda}_n}[\Psi_n(S_n)]=0$.\\

\textbf{(C2)}:  The marginal distribution $\overline{\Pi}_{n|L}$ of $L$ under $\overline{\Pi}_n$ is $\mu$. \\

\textbf{(C3)}: $\overline{\Pi}_{n}(L \in \partial\Psi_n(S_n))=1$ where $\partial\Psi_n$ denotes the subgradient of the convex function $\Psi_n$.\\

\textbf{(C4)}:  The marginal distribution  $\overline{\Pi}_{n|S_n}$ of $S_n$ under $\overline{\Pi}_n$, denoted $\nu_n$, is such that the density $ \frac{\partial \overline{\lambda}_{n}}{\partial \nu_n}$ is proportional to $E_{\overline{\Pi}_n}\left[ H'(LS_n-\Psi_n(S_n)) \mid S_n \right]$.\\

Conversely, one can always associate an equilibrium to a pair $(\Pi_n,\Psi_n)$ satisfying the four conditions. We prove in the subsection \ref{subsection:martingale_equiv} our first main result that claims the existence of a unique equivalent martingale measure as announced in the intorduction.

\begin{theorem_nonum}
Let $(\Pi_n,\Psi_n)$ be a pair satisfying the conditions \textbf{(C1)} to \textbf{(C4)}, $(p_q^n)_{q=1,\dots,n}$ be the price process that corresponds to $\Psi_n$ and $\Pi_{n | \omega}$ denote the marginal distribution of $\omega$ under $\Pi_n$. Then $\lambda_n$ is the unique probability equivalent to $\Pi_{n | \omega}$ such that the price process $(p_q^n)_{q=1,\dots,n}$ is a martingale when $(u_1,\dots,u_n)$ is $\lambda_n$-distributed.\\
\end{theorem_nonum}

This theorem justifies the following terminology:
the law of the price process under $\lambda_n$ is referred to as the martingale equivalent probability in the sequel of the paper. \\ 

We next deal in section \ref{section_existence} with the problem of existence of these reduced equilibria. This existence could possibly be proved by classical methods and approximations, using Nash-Glicksberg's theorem. One of the interest of our alternative proof is to introduce an operator $T_{\lambda}:\Delta(\mathbb{R})$ to $\Delta(\mathbb{R})$. This operator is the central tool to study the asymptotic properties of these equilibria.\\

To define this operator $T_{\lambda}$, we first focus on conditions \textbf{(C1)} to \textbf{(C3)}: for any measures $\lambda$ and $\nu$ on $\mathbb{R}$, there exists a unique pair $(\overline{\Pi}_{\nu},\Psi_{\nu,\lambda})$ satisfying \textbf{(C1)} with $\overline{\lambda}_n$ replaced by $\lambda$, \textbf{(C2)}, \textbf{(C3)} and such that the marginal distribution of $S_n$ under $\overline{\Pi}_{\nu}$ is $\nu$. Indeed, due to Fenchel lemma, condition \textbf{(C3)} can heuristically be interpreted by saying that $S_n$ is an increasing function\footnote{We remain very heuristical in our explanation at this point because $ \partial\Psi^{\sharp}_n$ is actually a correspondence and not a single valued function.} $g$ of $L$: $S_n=g(L)$. When $L$ is $\mu$-distributed, $g(L)$ must have distribution $\nu_n$. There is essentially a unique increasing function $g$ which satisfies that condition\footnote{Would $\nu$ have no atom, we would have $g(\ell)=F_{\nu}^{-1}(F_{\mu}(\ell))$, where $F_{\mu}$ and $F_{\nu}$ are the cumulative distribution functions of $\mu$ and $\nu$.} and we find therefore heuristically that $L=\Psi^{\prime}_n(S_n)=g^{-1}(S_n)$. This determines $\overline{\Pi}_n$ which is then the joint law of $(L,g(L))$ when $L$ is $\mu$-distributed. This also determines $\Psi_n$ up to a constant which can be found in a unique way to satisfy \textbf{(C1)}.\\

We are next seeking a measure $\nu$ that further satisfy \textbf{(C4)}. There exists a unique probability $\rho$ such that $\frac{\partial\rho}{\partial\nu}$ is proportional to $E_{\overline{\Pi}_{\nu}} [ H'(LS- \Psi_{\nu,\lambda}(S)) \mid S] $, since $H'>0$. Call $T_{\lambda}$ the map $\nu \rightarrow T_{\lambda}(\nu):= \rho$. With these notations, finding an equilibrium in $\overline{G}_n(\nu)$ is equivalent to find a measure $\nu$ satisfying the equation: 
\begin{equation}\label{eq_T}
T_{\overline{\lambda}_n}(\nu)=\overline{\lambda}_n 
\end{equation} 

The existence of equilibrium in the game $\overline{G}_n(\mu)$ is finally proved in section \ref{section_existence} by showing that the operator $T_{\overline{\lambda}_n}$ is onto the space of measures. We first prove the continuity of the operator $T_{\overline{\lambda}_n}$ in term of Wasserstein distance $W_2$. The onto property of $T_{\overline{\lambda}_n}$ results then from an application of the KKM theorem. Our second main results follows:

\begin{theorem_nonum}
There exists a Nash equilibrium in $\overline{G}_n(\mu)$, and therefore there also exists an equilibrium in the original game $G_n(\mu)$.
\end{theorem_nonum}

In the sequel of the paper, we focus on the equilibrium strategy of player 2, refereed to as the price process, and its properties when players play more and more frequently. In order to analyze the asymptotics of the price process, we have first to prove that any sequence $\nu_n$ converges (in Wasserstein distance $W_2$), where for all $n$, $\nu_n$ is a the solution of equation (\ref{eq_T})

Section \ref{section_convergence} of the paper is devoted this proof. Remember that $\overline{\lambda}_n$ is the law of $\frac{1}{\sqrt{n}}\sum_{q=1}^n u_q$ when $u_q$ are independent and centered. Due to the central limit theorem, $\overline{\lambda}_n$ converges to the normal law that we denote $\overline{\lambda}_{\infty}$. On the other hand, using a compactness argument, we can prove that any such sequence $(\nu_n)_{n \in \mathbb{N}}$ has an accumulation point $\nu$ satisfying 
\begin{equation}\label{rcbty}T_{\overline{\lambda}_{\infty}}(\nu)=\overline{\lambda}_{\infty}.
\end{equation} 

We then prove that there is a unique solution to this equation. This implies obviously that the sequence $(\nu_n)_{n \in \mathbb{N}}$ converges. To prove the above uniqueness result, we first prove that a solution to (\ref{rcbty}) should  also be a solution to the differential problem $(\mathcal{D})$ of Proposition \ref{prop:equiv} which results to have a unique solution (see Theorem \ref{Thm_uniqueness}). 

As shown in \ref{section_CMMV}, the convergence of $\nu_n$ implies the convergence of the law of the price process under the martingale equivalent measure, i.e. when $(u_1,\dots,u_n) \sim  \lambda_n$. More specifically, the discrete time price process $(p^n_1,\dots,p^n_n)$ can be represented by the continuous time price process $t \rightarrow p^n_{\lfloor nt \rfloor}$, were $\lfloor nt \rfloor$ is the integer part of $nt$. We first show that the processes $p^n_{\lfloor nt \rfloor}$ under the law $\lambda_n$ can be represented (Skorokhod embedding) on the natural filtration $\mathcal{F}$ of a Brownian motion on a probability space $(\tilde{\Omega},\mathcal{F},\tilde{P})$. Indeed, there exists a sequence of variables $(\tilde{p}_q^n)_{q \in \{1,\dots,n \}}$ and an increasing sequence of stopping times $(\tau_q^n)_{q \in \{1,\dots,n \}}$ such that $\tilde{p}_q^n$ is $\mathcal{F}_{\tau_n^q}$-mesurable, and has the same distribution as $p_q^n$ when $(u_1,\dots,u_n) \sim \lambda_n$. We show in Theorem \ref{thm:equiv} that $\tilde{p}^n_{\lfloor nt \rfloor}$ converges in finite dimensional distribution to a limit process $Z_t$ defined on the space $(\tilde{\Omega},\mathcal{F},\tilde{P})$ and that results to be a CMMV.  

This is our third important result:

\begin{theorem_nonum}
Let $(p_q^n)_{q \in \{1,\dots,n\}}$ be an equilibrium strategy in the reduced game $\overline{G}_n(\mu)$. As $n$ goes to infinity, the law of the stochastic process $t \rightarrow p^n_{\lfloor nt \rfloor}$ under $\lambda_n$ (i.e. when $(u_1,\dots,u_n)\sim \lambda_n$) converges to the law of a continuous martingale of maximal variation $Z_t$.
\end{theorem_nonum}

We also show the convergence of the historical law of $p^n_{\lfloor nt \rfloor}$ (i.e. when $(u_1,\dots,u_n)$ are distributed according to the equilibrium strategy of player 1) and we finally prove that this limit law is equivalent to the law of the CMMV $Z_t$.\\

\section{Reduced equilibrium} \label{section_reduced_eq}

\begin{definition}
The reduced game $\overline{G}_n(\mu)$ is the game where player 1 does not observe player 2's actions and player 2 is not allowed to randomize his moves (he only uses pure strategies). 
\end{definition}

The aim of this section is to prove Proposition \ref{prop:reduced_to_general}. It states that any equilibrium in $\overline{G}_n(\mu)$ is an equilibrium in $G_n(\mu)$. 

In this paper a pure strategy of player 2 will be denoted $\overline{p}$. Such a strategy $\overline{p}$ is thus a vector $(\overline{p}_1,\dots,\overline{p}_n)$ where $\overline{p}_q$ is a map $\mathbb{H}_{q-1} \rightarrow \mathbb{R}$. $\overline{p}_q(h_{q-1})$ denotes then the deterministic action taken by player 2 after history $h_{q-1}$. Remark however that since player 2 does not randomize before stage $q$, the action he will take at stage $q$ is just a deterministic function of previous moves of player 1. Therefore, in this paper, a pure strategy of player 2 will be considered as a sequence $(\overline{p}_1,\dots,\overline{p}_n)$ where $\overline{p}_q$ is a function $\{-1,+1\}^{q-1} \rightarrow \mathbb{R}$.\\

The intuition behind Lemma \ref{lemma:pure_reduced_p2} hereafter is that to any mixed strategy of player 2, he will prefer the corresponding "average" pure strategy. More precisely, let $\sigma$ be a reduced strategy of player 1 and let $\tau$ be any strategy of player 2. Since player 1 does not observe player 2's move when using $\sigma$, one can assume that he picks his actions $u_1,\ldots,u_n$ after observing $L$ and before player 2's first move. $(\mu,\sigma)$ induces thus a probability on $(L,u_1,\ldots, u_n)$.
Player 2's strategy $\tau$ can then be viewed as a device to randomly chose $(p_1,\ldots, p_n)$ once $(u_1,\ldots ,u_n)$ has been selected: one first select $p_1$ with the lottery $\tau_1$, then one selects $p_2$ with the lottery $\tau_2(u_1,p_1)$ and so on. Therefore $\tau$ determines the conditional law of $p_q$ given $(u_1,\ldots, u_{q-1})$ and this conditionnal law does not depend on $\sigma$. Let $f^\tau_q(u_1,\dots,u_{q-1})$ denote the  expected value of this conditional law.
Note that $p_\tau:=(f^\tau_1,\ldots,f^\tau_n)$ is then a pure strategy for player 2 since it does not depend on $\sigma$. We have thus: 
$$E_{\mu, \sigma, \tau}[ p_q  \mid u_1, \dots , u_{q-1}]=f^\tau_q(u_1,\dots,u_{q-1}).$$
Observe also that the law of $p_q$ given $(u_1, \dots , u_{q-1})$ is just the law of $p_q$ given $(u_1, \dots , u_{n},L)$.
Indeed, given $(u_1, \dots , u_{q-1})$, $p_q$ must clearly be independent of $(u_q, \dots , u_{n},L)$ since $(u_q, \dots , u_{n},L)$ was chosen before $p_q$ and $\tau_q$ just depends on $(u_1, \dots , u_{q-1})$.
Therefore $$E_{\mu, \sigma, \tau}[ p_q  \mid u_1, \dots , u_{n},L]=f^\tau_q(u_1,\dots,u_{q-1}).$$

We now compare the payoffs induced by a strategy $\tau$ with those induced by the corresponding strategy $p_{\tau}$.

\begin{lemma}\label{lemma:pure_reduced_p2}
For any $\tau$ a strategy of player 2, $p_{\tau}$ is such that for all reduced strategy $\sigma$ of player :
\[
\begin{cases}
g_2(\sigma,\tau) \geq g_2(\sigma,p_{\tau})\\
g_1(\sigma,\tau)=g_1(\sigma,p_{\tau})
\end{cases} 
\]

\end{lemma}

\begin{proof}

To simplify notations, the expectation $E_{\mu,\sigma,\tau}$ is denoted $E$.

\[g_2(\sigma, \tau)= E [H( \sum_{q=1}^n u_q (L - p_q ))]= E [ E [H( \sum_{q=1}^n u_q (L - p_q )) \mid u_1, \dots , u_n, L]]\]

We now apply Jensen's inequality to the convex function H, and take into account the fact that $u_q$ and $L$ are $(u_1, \dots , u_n, L)$-measurable:

\[\begin{array}{rcl}g_2(\sigma, \tau) &\geq  & E [ H ( E [ \sum_{q=1}^n u_q (L - p_q ) \mid u_1, \dots , u_n, L])]\\& = & E [ H (  \sum_{q=1}^n u_q (L -  E [ p_q  \mid u_1, \dots , u_n,L]))]\\&= &E [ H (  \sum_{q=1}^n u_q (L -  f^\tau_q(u_1,\dots,u_{q-1})))]\\&=&g_2(\sigma,p_{\tau})\end{array}\]

Similarly, we have:

\begin{align*}
g_1(\sigma,\tau)=& E_{\mu,\sigma,\tau}[u_q(L-p_q)]\\
=&E_{\mu, \sigma,\tau} [ \sum_{q=1}^n u_q (L - E_{\tau}[p_q \mid u_1,\dots,u_{q-1}])]\\
=& E_{\mu, \sigma,\tau} [ \sum_{q=1}^n u_q (L - f^\tau_q(u_1,\dots,u_{q-1}) )]\\
=&g_1(\sigma,p_{\tau})
\end{align*}

\end{proof}

We also will need the following lemma:
\begin{lemma}\label{lemma:reduced_P1}
Let $\overline{p}$ be a pure strategy of player 2 and $\sigma$ any strategy of player 1 (even non reduced). Then, there exists a reduced strategy of player 1 denoted $ \tilde{\sigma}_{(\sigma,\overline{p})}$ which gives him the same payoff as $\sigma$ against $\overline{p}$, i.e. : \[   g_1(\tilde{\sigma}_{(\sigma,\overline{p})},\overline{p})=g_1(\sigma,\overline{p})  \] \end{lemma}

\begin{proof}

The strategy $\sigma$ is not reduced, so $\sigma_q$ depends on $(u_1,\dots,u_{q-1},\overline{p}_1,\dots,\overline{p}_{q-1})$. But player 2 is completely deterministic since he uses strategy $\overline{p}$. Therefore he plays action $p_q=\overline{p}_q(u_1,\dots,u_{q-1})$, and the whole history $p_1,\dots,p_{q-1}$ is just a deterministic function of $u_1,\dots,u_{q-2}$. In the arguments of $\sigma_q$, we can replace $p_1,\dots,p_{q-1}$ by this function and we get in this way:

\[\tilde{\sigma}_{(\sigma,p),q} (u_1,\dots,u_{q-1}):=\sigma(u_1,\dots,u_{q-1},\overline{p}_1,\dots,\overline{p}_{q-1}(u_1,\dots,u_{q-2})) \]

which is a reduced strategy and clearly: $g_1(\tilde{\sigma}_{(\sigma,\overline{p})},\overline{p})=g_1(\sigma,\overline{p})$
\end{proof}

We are now ready to prove the main result of this section:

\begin{proposition}\label{prop:reduced_to_general}
If $(\sigma^{\star},\overline{p}^{\star})$ is an equilibrium in $\overline{G}_n(\mu)$, then $(\sigma^{ \star},\overline{p}^{\star})$ is an equilibrium in $G_n(\mu)$.
\end{proposition}

\begin{proof}

For all player 2's strategy $\tau$ in $G_n(\mu)$, we have:

 \[g_2(\sigma^{\star},\overline{p}^{\star}) \leq g_2(\sigma^{\star},p_{\tau}) \leq g_2(\sigma^{\star},\tau)\]
where $p_{\tau}$ is defined as above. Indeed the first inequality just indicates that the pure strategy $p_{\tau}$ is not a profitable deviation from the equilibrium strategy $\overline{p}^{\star}$ in $\overline{G}_n(\mu)$. The second inequality comes from Lemma \ref{lemma:pure_reduced_p2}.

Let $\sigma$ be any strategy of player 1. With the notation of Lemma \ref{lemma:reduced_P1} we get:

\[ g_1(\sigma,\overline{p}^{\star})=g_1(\tilde{\sigma}_{(\sigma,\overline{p}^{\star})},\overline{p}^{\star}) \leq g_1(\sigma^{\star},\overline{p}^{\star})
\]

where the inequality follows from the fact that $\tilde{\sigma}$ is a reduced strategy and can thus not be a profitable deviation from $\sigma^{\star}$ for player 1.
\end{proof}

Based on the previous proposition, equilibria in $\overline{G}_n(\mu)$ will be referred to as the reduced equilibria in $G_n(\mu)$. In the sequel of this paper, we will only focus  on the reduced equilibria of $G_n(\mu)$.

\section{Characterisation of equilibrium}\label{section_conditions}~~

In subsection \ref{subsection:Alt_repr} we give an other representation of the strategy spaces in $\overline{G}_n(\mu)$. This representation is needed in subsection \ref{subsection:charac}, where we provide necessary and sufficient equilibrium conditions.\\

Finally, in subsection \ref{subsection:martingale_equiv}, we prove that the price process posted by player 2 in any equilibrium  is a martingale when the past actions of player 1 are uniformly distributed. Moreover, we prove that the uniform distribution is the only probability on player 1's actions under which the price process becomes a martingale.

\subsection{Alternative representation of the strategy spaces}\label{subsection:Alt_repr}~~

When playing a reduced strategy player 1 does not observe player 2's actions and we can therefore assume that he selects his actions after getting the information $L$ and before the first move of player 2. Thus, joint with $\mu$, a reduced strategy $\sigma$ induces a joint law $\Pi_n$ on $(L,\omega)$ where $\omega=(u_1,\dots,u_n)$ belongs to $\Omega_n:=\{ -1,+1 \}^n$. The marginal $\Pi_{n | L}$ of $\Pi_n$ on $L$ is clearly $\mu$. We can further recover the strategy $\sigma$ from $\Pi_n$ computing the conditional probabilities given $L$. Therefore the player 1's strategy space may be viewed as the set of probabilities $\Pi_n$ in $\Delta(\mathbb{R}\times \Omega_n)$ such that $\Pi_{n | L}=\mu$.\\

We first prove that player 2's strategy space in $\overline{G}_n(\mu)$ can be identified with a set $\mathbb{X}_n$ of random variables. \\

Let us consider the set of pure strategies $\mathcal{P}$ of player 2. If $\overline{p} \in \mathcal{P}$, then $\overline{p}_q$ is  a function $\Omega_n \rightarrow \mathbb{R}$ which is measurable with respect to $(u_1,\dots,u_{q-1})$. Note that the strategy $\overline{p}$ only appears in the payoff functions (see equations \ref{formula_payoff_p1} and \ref{formula_payoff_p2}) thought the quantity  $ X_{n,\overline{p}}(\omega):= \frac{1}{\sqrt{n}} \sum_{q=1}^{n} u_q \overline{p}_q( \omega)$. We can therefore identify the strategy space of player 2 with the set $\mathbb{X}_n:=\{X_{n,\overline{p}} | \overline{p} \in \mathcal{P} \} \subset \mathbb{R}^{\Omega}$.

Next lemma characterizes this set. Let $\lambda_n$ be the uniform probability on $\Omega_n$. Under $\lambda_n$, $(u_q)_{q=1,\dots,n}$ are mutually independent, and have zero expectation. We denote $L^1(\lambda_n)$ the set $L^1(\Omega_n,\mathcal{P}(\Omega_n),\lambda_n)$ which is just $\mathbb{R}^{\Omega_n}$ since $\Omega_n$ is finite.

\begin{lemma}\label{lemma:X_n}
$\mathbb{X}_n=\{ X \in L^1(\lambda_n)\mid E_{\lambda_n} [X] = 0 \}$
\end{lemma}

\begin{proof}
Let $X \in \mathbb{X}_n$. Then $X=X_{n,\overline{p}}$ for some $\overline{p} \in \mathcal{P}$. Since $\Omega_n$ is a finite set, $X$ as a map from $\Omega_n$ to $\mathbb{R}$ belongs to $L^1(\lambda_n)$. Moreover, using that $\overline{p}_q$ is $(u_1,\dots,u_{q-1})$ measurable:

\[E_{\lambda_n}\left[\frac{1}{\sqrt{n}} \sum_{q=1}^{n} u_q p_q\right]=E\left[\frac{1}{\sqrt{n}} \sum_{q=1}^{n} E_{\lambda_n}[   u_q p_q| u_1,\dots,u_{q-1}]\right]=E\left[\frac{1}{\sqrt{n}} \sum_{q=1}^{n} p_q E_{\lambda_n}[   u_q ]\right]=0
\]

We thus have proved that $\mathbb{X}_n \subseteq \{ X \in L^1(\lambda_n) \mid E_{\lambda_n} [X] = 0 \}$.\\
Assume next that $X \in L^1(\lambda_n)$ is such that $ E_{\lambda_n} [X] = 0$.  For $k \in \{ 1, \dots,  n-1  \}$, we denote $X^k(u_1,\dots,u_k):=E_{\lambda_n}[X \mid u_1, \dots, u_k]$. Let $ \textbf{1}_{\{u_k=1\}}$ denotes the random variable that takes the value $1$ if $u_k=1$ and $0$ otherwise. An easy computation shows that $ \textbf{1}_{\{u_k=1\}}=\frac{u_q+1}{2}$. One gets therefore

\[  X^{k}(u_1,\dots,u_{k}) =  \textbf{1}_{\{u_k=1\}} X^{k}(u_1,\dots,u_k-1,1)  +    \textbf{1}_{\{u_k=-1\}} X^{k}(u_1,\dots,u_k-1,-1) \]

   \[ = \frac{  u_{k} +1   }{2}  X^{k}(u_1,\dots,u_{k-1},1)   + \frac{1-u_{k}}{2}  X^{k}(u_1,\dots,u_{k-1},-1) \]

\[ =  \frac{u_{k} \overline{p}_{k}(\omega)}{\sqrt{n}}  + \frac{ X^{k}(u_1,\dots,u_{k-1},1) + X^{k}(u_1,\dots,u_{k-1},-1)}{2}, \]

where: 

\begin{equation}\label{eq_X_p_to_p}
\overline{p}_{k}(\omega): = \frac{ X^{k}(u_1,\dots,u_{k-1},1)- X^{k}(u_1,\dots,u_{k-1},-1)}{2 / \sqrt{n}} 
\end{equation}  

Now observe that $$\begin{array}{rcl}X^{k-1}(u_1,\dots,u_{k-1})&=&E_{\lambda_n}[X^k | u_1,\dots,u_{k-1}]\\&=&\frac{ X^{k}(u_1,\dots,u_{k-1},1) + X^{k}(u_1,\dots,u_{k-1},-1)}{2}\end{array}$$

Therefore $X^{k}(u_1,\dots,u_{k}) - X^{k-1}(u_1,\dots,u_{k-1}) = \frac{u_{k} \overline{p}_{k}(\omega)}{\sqrt{n}}$.

Summing up those equalities for $k=1$ to $n$, we get:

\[ X^n(u_1,\dots,u_{n})= \frac{\sum_{k=1}^{n}  u_k \overline{p}_{k}(\omega)}{\sqrt{n}} + X_0 \]

But $X^n(u_1,\dots,u_{n}) = X$ and $X_0 = E_{\lambda}(X)=0$. We get thus:

\[ X =   \frac{\sum_{k=1}^{n} u_k \overline{p}_{k}(\omega)}{\sqrt{n}} = X_{n,\overline{p}},   \]

for the strategy $\overline{p}$ defined in (\ref{eq_X_p_to_p}).
\end{proof}

Let us make more precise the relation between $X$ and the strategy $\overline{p}$ such that $X=X_{n,\overline{p}}$.

\begin{proposition}\label{prop:X_to_p}
Let $X \in \mathbb{X}_n$. There exists a unique pure reduced strategy $\overline{p}$ such that $X= X_{n,\overline{p}}$. 
Moreover, we have the explicit formula:
\begin{equation}\label{equation_p_q}
 \overline{p}_q(u_1,\dots,u_{q-1}) = \sqrt{n} E_{\lambda_n} [ u_q X  \mid u_1, \dots, u_{q-1} ]
\end{equation}

\end{proposition}

\begin{proof}

Let $\overline{p}_j$ be $(u_1,\dots,u_{j-1})$-measurable. Then observe that if $j<q$: 
\[E_{\lambda_n}[\overline{p}_j u_q u_j | u_1,\dots,u_{q-1}]=\overline{p}_j u_j E_{\lambda_n}[u_q | u_1,\dots,u_{q-1}]=0\] 

On the other hand, if $j>q$, 
\begin{align*}
E_{\lambda_n}[\overline{p}_j u_q u_j | u_1,\dots,u_{q-1}]= &E_{\lambda_n}[E_{\lambda_n}[\overline{p}_j u_q u_j|u_1,\dots,u_{j-1}] | u_1,\dots,u_{q-1}]\\
= &E_{\lambda_n}[\overline{p}_j u_q  E_{\lambda_n}[u_j|u_1,\dots,u_{j-1}] | u_1,\dots,u_{q-1}] \\
= &0
\end{align*}

We get thus $E_{\lambda_n}[\overline{p}_j u_q u_j | u_1,\dots,u_{q-1}]= \overline{p}_q$ if $j=q$ and $0$ otherwise. Let now $X$ be in $\mathbb{X}_n$. According to the previous lemma, $X=X_{n,\overline{p}}$ for some $\overline{p}$. We can therefore write $E_{\lambda_n} [ u_q X  \mid u_1, \dots, u_{q-1} ]= E_{\lambda_n} [ u_q  \frac{\sum_{i=1}^n \overline{p}_i u_i}{\sqrt{n}}  \mid u_1, \dots, u_{q-1} ]=\frac{1}{\sqrt{n}}\sum_{i=1}^n E_{\lambda_n} [ \overline{p}_i u_q   u_i \mid u_1, \dots, u_{q-1} ]=\frac{\overline{p}_q}{\sqrt{n}}$ as announced. \end{proof}

We can now reformulate the completely reduced game $\overline{G}_n(\mu)$ as follow: player 1 selects $\Pi_n \in \Delta(\Omega_n \times\mathbb{R})$ such that $\Pi_{n | L}=\mu$. Simultaneously player 2 chooses $X \in \mathbb{X}_n$.

The payoff functions are now given by the formula:

\[
\begin{cases}
g_1(\Pi_n,X)= E_{\Pi_n} [ L S_n - X ] \\
g_2(\Pi_n,X)= E_{\Pi_n} [H (L S_n - X )] 
\end{cases}
\]

where $S_n(\omega) =\frac{1}{\sqrt{n}}\sum_{q=1}^{n} u_q$.

\subsection{Characterization of equilibrium strategies in $\overline{G}_n(\mu)$}
\label{subsection:charac}~~

The main result of this subsection is Corollary \ref{Cor:C1234} that provides necessary and sufficient conditions for a pair $(\Pi_n^{\star}, X^{\star})$ to be an equilibrium in $\overline{G}_n(\mu)$.

As a first step in the proof of that result, we prove in Proposition \ref{prop:allhistory} that any history $\omega=(u_1,\dots,u_n)$ has a positive probability at equilibrium. We express this property saying  that the equilibrium strategy $\Pi_n^{\star}$ of player 1 is completely mixed.

We next argue in Proposition \ref{prop:X_is_convex} that if a strategy $X$ of player 2 is such that there exists a completely mixed best response of player 1, then $X$ has a very particular form: $X(\omega)$ is a convex function of $S_n(\omega)$.

The next result is Proposition \ref{FOC_for_L}. It claims that if player 2 plays a strategy $X=\Psi_n(S_n)$ for a convex function $\Psi$ then $\Pi_n$ is a best response to $X$ if and only if $\Pi_n(L \in \partial \Psi_n(S_n))=1$, where $\partial \Psi_n$ is the subgradient of the convex function $\Psi_n$, as defined in equation (\ref{def:subgrad}).

Finally the first order condition for Player 2's strategy $\Psi(S_n)$ are derived in Proposition \ref{density_necessary}.

\begin{proposition}\label{prop:allhistory}

If player 2 has a best reply to a strategy $\Pi_n^{\star}$ of player 1 in $\overline{G}_n(\mu)$ then $\Pi_n^{\star}$ is completely mixed.

\end{proposition}

\begin{proof}
$\Pi_n^{\star}$ is a probability on $(L,\omega)$ where $\omega=(u_1,\dots,u_n)$. It induces therefore a marginal distribution on $\omega_q=(u_1,\dots,u_q)$. Denote $\Gamma_q$ the set of $\omega_q$ such that $\Pi_n^{\star}(\omega_q)>0$. We want to prove that $\Gamma_n=\Omega_n$. Assume on the contrary that $\Gamma_n \neq \Omega_n$. We can then define $q^{\star}$ as the smallest $q$ such that $\Gamma_q \neq \Omega_q:=\{-1,+1 \}^q$. There is then a history $(u_1,\dots,u_{q^{\star}-1}) \in \Gamma_{q^{\star}-1}$ such that one of the histories $(u_1,\dots,u_{q^{\star}-1},1)$ or $(u_1,\dots,u_{q^{\star}-1},-1)$ does not belong to $\Gamma_q^{\star}$. Whence, this history $(u_1,\dots,u_{q^{\star}-1})$ has a positive probability under $\Pi_n^{\star}$ and is followed by a deterministic move of player 1 at stage $q^{\star}$. But after observing this history, player 2 could increase his benefit by posting a higher or  lower price according to the forecoming deterministic move of player 1. This contradicts the hypothesis that there is a best reply against $\Pi_n^{\star}$. Therefore, assuming $\Gamma_n \neq \Omega_n$ leads to a contradiction.
\end{proof}
The following  notions are classical and useful concept to deal with convex functions:
 The subgradient $\partial\Psi(s)$  of a convex function $\Psi$ at $s$ is defined as:
 
 \begin{equation}\label{def:subgrad}
\partial\Psi(s)=\{ \ell | \forall z: \Psi(z) \geq \Psi(s) + \ell (z-s) \}. 
 \end{equation}
 The Fenchel transform of $\Psi$ is defined as the convex function $\Psi^\sharp( x)$:

\begin{equation}\label{def:fenchel}
\Psi^\sharp(x):=\sup_{s\in\Bbb{R}} xs-\Psi(s).
\end{equation}
As well known (see \citet{rock}), if $\Psi$ is lower semi-continuous, then $\Psi=(\Psi^\sharp)^\sharp$.
Furthermore, we have the following equivalence due to Fenchel:

\begin{equation}\label{eq:fenchel}
x\in\partial \Psi(s)\Leftrightarrow xs=\Psi(s)+\Psi^\sharp(x)\Leftrightarrow s\in\partial \Psi^\sharp(x).
\end{equation}

We are now ready to state our next result:
\begin{proposition}\label{prop:X_is_convex} 
If player 1 has a completely mixed best reply $\Pi_n^{\star}$ to a strategy $X^{\star}$ of player 2 in $\overline{G}_n(\mu)$, then  $X^{\star} = \Psi_n(S_n(\omega))$ where $\Psi_n$ is a convex function such that $E_{\lambda_n}[\Psi_n(S_n)]=0$.
\\We further have $ \Psi_n=A^\sharp$ where \begin{equation}\label{eq_def_of_A_L}
A(\ell):=\displaystyle \max_{\omega' \in \Omega_n}  \ell S_n(\omega')-X^{\star}(\omega').
\end{equation} 
\end{proposition}

\begin{proof}

Suppose that player 2 is playing $X^{\star}$ and player 1 wants to maximize his payoff $E_{\Pi_n}[LS-X^{\star}]$. After observing $L=\ell$, he will select an history $\omega \in V_{\ell}$ where $V_{\ell}$ is the set of $\omega'$ that solve the maximization problem $A(\ell)$ in equation (\ref{eq_def_of_A_L}).

Therefore \begin{equation}\label{equa_omega_in_V_L}
\Pi_n^{\star}(\omega \in V_L)=1.
\end{equation}

Since all history $\omega$ has a positive probability under $\Pi_n^{\star}$ we conclude that the set of values $\ell$ such that $\omega \in V_{\ell}$ can not be empty. Otherwise $\omega$ would never be selected by player 1 and would have zero probability under $\Pi_n^{\star}$.

Now remark that it follows from the definition of $A$ that for all $\ell$ and for all $\omega$: 

\begin{equation}\label{eq_ineq}
 A(\ell) \geq \ell S_n(\omega) - X^{\star}(\omega)
\end{equation}

Therefore, for all $\omega$, for all $\ell$:

$$X^{\star}(\omega) \geq  \ell S_n(\omega) - A(\ell)  $$

and thus for all $\omega$:

$$ X^{\star}(\omega) \geq   \displaystyle \sup_{\ell \in \mathbb{R}} ~~ \ell S_n(\omega) - A(\ell).$$

As observed above, for all $\omega$, the set of $\ell$ such that $\omega \in V_{\ell}$ is not empty.

For those $\ell$, inequality (\ref{eq_ineq}) is an equality, and thus:

$$X^{\star}(\omega) =   \displaystyle \sup_{\ell \in \mathbb{R}} ~~ \ell S_n(\omega) - A(\ell).$$

We get therefore $X^{\star}(\omega)=\Psi_n(S_n(\omega))$ with $\Psi_n(s)=A^\sharp(s)$. Observe that as supremum of affine functions of $s$, the map $s \rightarrow \Psi_n(s)$ is convex.

Finally, since $X \in \mathbb{X}_n$ we get with Lemma \ref{lemma:X_n} that $E_{\lambda_n}[\Psi_n(S_n(\omega))]=0$.
\end{proof}
With this notation, we have the following result:
\begin{proposition}\label{FOC_for_L} Consider the strategy $X^{\star}(\omega)=\Psi_n(S_n(\omega))$ of last proposition. A strategy $\Pi_n$ of player 1 is a best response to $X^{\star}$ if and only if:
$$\Pi_n(L \in \partial\Psi_n(S_n))=1.$$
\end{proposition}
\begin{proof}
As explained in the beginning of the previous proof,  $\Pi_n$ is a best response to $X^{\star}$ if and only if $\Pi_n(\omega \in V_L)=1$. It follows from the definition of $V_\ell$ that  $\omega \in V_{\ell}$ if and only if $A(\ell)=S_n(\omega)\ell-X(\omega)$. On the other hand, it follows from the definition (\ref{eq_def_of_A_L}) of $A$ that for all $r$, $A(r) \geq S_n(\omega)r-X(\omega)$. Combining these two relations, we have that $A(r) \geq S(\omega)(r-\ell)+A(\ell)$ and thus $S(\omega) \in \partial A(\ell)$ or equivalently $\ell \in \partial A^\sharp(S_n(\omega))= \partial\Psi_n(S_n(\omega))$. Therefore $\omega \in V_{\ell}$ if and only if $\ell \in \partial\Psi_n(S_n(\omega))$, and the proposition is prooved.\end{proof}

The next proposition expresses the first order conditions of player 2 optimization problem. $\Pi^{\star}_{n | \omega}$ just denotes the marginal distribution of $\omega$ under $\Pi^{\star}_{n}$.

\begin{proposition}\label{density_necessary}

A strategy $X^{\star}$ is a best reply to a strategy $\Pi_n^{\star}$ of player 1, if and only if $\lambda_n$ has a density with respect to $\Pi_{n|\omega}^{\star}$ given by the formula:

\[ \frac{d\lambda_n}{d\Pi^{\star}_{n | \omega}}=\alpha_n E_{\Pi_n^{\star}}[ H'(LS_n-X_n^{\star}) \mid \omega  ]
\]
for a constant $\alpha_n$.
\end{proposition}

\begin{proof}

Suppose that $X^{\star}$ is a best reply to a strategy $\Pi_n^{\star}$ of player 1. Then  $X^{\star}$ is a solution to the minimization problem of player 2:

\[ \displaystyle \min_{X \in \mathbb{X}_n}E_{\Pi_n^{\star}} [ H ( LS_n-X)]. \]

Note that the map $X \rightarrow E_{\Pi_n^{\star}} [ H ( LS_n-X)]$ is convex in $X$ and we are in front of a convex minimization problem. In such a problem the first order conditions are both necessary and sufficient. We get these first order conditions considering for fixed $\delta \in \mathbb{X}_n$ the map
$G:\epsilon \in \mathbb{R} \rightarrow G(\epsilon):= E_{\Pi_n^{\star}} ( H ( LS_n-X^{\star}+\epsilon \delta))$. This map must reach a minimum at $\epsilon=0$.

Observe now that $H$ is $C^1$ and so is $G$. We get then $G'(0)=E_{\Pi_n^{\star}} [ H' ( LS_n-X^{\star})\delta]$, and therefore, for all $\delta \in \mathbb{X}_n$:

\[ E_{\Pi_n^{\star}} [ H' ( LS_n-X^{\star})\delta]=0
\]

Since $\delta$ is just a function of $\omega$, this equality can also be written as: 

\[0=E_{\Pi_n^{\star}} [ E[ H'(LS_n-X_n^{\star})\delta | \omega]]=E_{\Pi_n^{\star}} [ \delta Y_n]
\]

where $Y_n(\omega):= E_{\Pi_{n}^{\star}}[ H'(LS^n-X^{n, \star}) \mid \omega  ]$. $Y_n(\omega)>0$ because $H'>\epsilon>0$ according to \textbf{A2}.

Since $\lambda_n$ puts a positive weight on every history, $\Pi_{n | \omega}^{\star}$ is absolutely continuous with respect to $\lambda_n$ and has a density $y_n= \frac{d\Pi^{\star}_{n | \omega}}{d\lambda_n}$.

We can rephrase previous conditions as: for all $\delta \in \mathbb{X}_n$,

\[   E_{\lambda_n} [ y_n  Y_n  \delta] = 0
\]

This relation can interpreted as an orthogonality relation in $L^2(\lambda_n)$ with the scalar product $\langle A,B \rangle:= E_{\lambda_n} [A B]$. The space $\mathbb{X}_n$ must then be orthogonal to $y_n Y_n$. But Lemma \ref{lemma:X_n} shows that $\mathbb{X}_n= \{1\}^{\perp}$. Therefore $y_nY_n$ is co-linear with $1$: it is equal to a positive constant that we denote $\frac{1}{\alpha_n}$. 

Since $y_n= \frac{d\Pi^{\star}_{n | \omega}}{d\lambda_n}>0$, $\lambda_n$ is absolutely continuous with respect to $\Pi_{n|\omega}^{\star}$ and we get $\frac{d\lambda_n}{d\Pi^{\star}_{n | \omega}}=\frac{1}{y_n}=\alpha_n Y_n$.
\end{proof}

\begin{corollary}\label{Cor:C1234}
A pair of strategy $(\Pi_n^{\star},X^{\star})$ is an equilibrium in $\overline{G}_n(\mu)$ if and only if $ \forall \omega: X^{\star}(\omega)=\Psi_n^{\star}(S_n(\omega))$, where
$\Psi_n^{\star}$ is a convex function that jointly satisfy with $\Pi_n^{\star}$ the following conditions \textbf{(C1)},\textbf{(C2)},\textbf{(C3)},\textbf{(C4)}.
\end{corollary}

\[
\begin{cases}
\textbf{(C1)} ~~ \Psi_n^\ast \text{ is such that } E_{\lambda_n}[\Psi_n^\ast(S_n(\omega)]=0 \\
\textbf{(C2)} ~~ \Pi_{n | L}^{\star}=\mu \\
\textbf{(C3)} ~~ \Pi_n^{\star}(L \in \partial\Psi_n(S_n))=1\\
\textbf{(C4)} ~~ \frac{\partial\lambda_n}{\partial \Pi_{n | \omega}^{\star}}=  \alpha_n E_{\Pi_n^{\star}}[ H'(LS_n-\Psi_n^{\star}(S_n)) | \omega  ] \text{ where $\alpha_n$ is a constant} \\
\end{cases}
\]

From now on, a pair $(\Pi_n^{\star},\Psi^{\ast}_n)$ satisfying \textbf{(C1)},\textbf{(C2)},\textbf{(C3)},\textbf{(C4)} will be referred to as an equilibrium in $\overline{G}_n(\mu)$ (instead of the pair $(\Pi_n^{\star},X^\ast)$, with $X^\ast:=\Psi^{\ast}_n(S_n)$.)

Remember that according to the results of section \ref{subsection:Alt_repr},  such a pair $(\Pi_n^{\star},\Psi^{\ast}_n)$ fully describes  a pair  of equilibrium strategies $(\sigma^\ast,\tau^\ast)$ in the original game $G_n(\mu)$.

\subsection{The price process and the martingale equivalent measure}
\label{subsection:martingale_equiv}~~

Before proving the existence of equilibrium in section \ref{section_existence}, let us emphasize that the above characterization of equilibrium implies that under an appropriate equivalent measure the price process is a martingale.

Consider an equilibrium $(\Pi_n^{\star},\Psi_n^{\star})$,  and denote $\overline{p}_1,\overline{p}_2(u_1),\dots,\overline{p}_n(u_1,\dots,u_{n-1})$ the corresponding pure strategy of player 2. 

 When $(u_1,\dots,u_n)$ are randomly selected by player 1 with lottery $\Pi_n^{\star}$, the law of this process $\overline{p}$ is called the historical law. We now prove that if $(u_1,\dots,u_n)$ are selected by the lottery $\lambda_n$, the process is a martingale.

\begin{theorem}\label{price_is_mart}
The price process $(\overline{p}_q^n)_{q=1,\dots,n}$ is a martingale under the probability $\lambda_n$. \\
\end{theorem}
\begin{proof}
With equation (\ref{equation_p_q}) we have:

\begin{align}
\overline{p}^n_q(u_1,\dots,u_{q-1}) &= \sqrt{n} E_{\lambda_n} [ u_q X^{\star}  \mid u_1, \dots, u_{q-1} ]  \nonumber \\ 
&= \sqrt{n} E_{\lambda_n} [ u_q \Psi_n(S_n)  \mid u_1, \dots, u_{q-1} ] \nonumber \\
&= \sqrt{n} E_{\lambda_n} [ u_n \Psi_n(S_n)  \mid u_1, \dots, u_{q-1} \label{eq:price}]
\end{align}

The last equality follows from the fact that, conditionally to $u_1,\dots,u_{q-1}$, the vector $(u_q,S_n)$ and $(u_n,S_n)$ have the same law under $\lambda_n$. The price process $\overline{p}^n$ is written as a conditional expectation of a terminal variable with respect to an increasing sequence of $\sigma$-algebras. It is then a martingale under the probability $\lambda_n$. 
\end{proof}

We further aim to prove that $\lambda_n$ is the unique probability on $\Omega_n$ that makes the price process a martingale.

\begin{theorem}\label{thm:unique_equiv_measure}
 $\lambda_n$ is the unique probability on $\Omega_n$ that makes the price process $(\overline{p}_q)_{q=1,\dots,n}$ a martingale.
\end{theorem} 

\begin{proof}
Indeed, let $\tilde{\lambda}_n$ be a probability on $\Omega_n$ under which the price process is a martingale. 

We find with the similar computation as that made to get equation (\ref{equation_p_q}) that:  

\[ \overline{p}_q(u_1,\dots,u_{q-1})=\frac{u_q+1}{2}\overline{p}_q(u_1,\dots,u_{q-2},1)+\frac{u_q-1}{2}\overline{p}_q(u_1,\dots,u_{q-2},-1)\]

Since $\overline{p}$ is a martingale under $\lambda_n$, we find $\frac{\overline{p}_q(u_1,\dots,1)+\overline{p}_q(u_1,\dots,-1)}{2}=\overline{p}_{q-1}(u_1,\dots,u_{q-2})$

And thus 

\[ \overline{p}_q(u_1,\dots,u_{q-1})=\overline{p}_{q-1}(u_1,\dots,u_{q-2})+c_q(u_1,\dots,u_{q-2})u_{q-1}
\]
where $c_q=\frac{\overline{p}_q(u_1,\dots,1)-\overline{p}_q(u_1,\dots,-1)}{2}$. Lemma \ref{lemma:p_q_is_increasing} proved in Annex \ref{annexe_6}, indicates that $c_q>0$.

So if $\overline{p}$ is a martingale under $\tilde{\lambda}_n$, we must have for all $q$: $E_{\tilde{\lambda}_n}[u_{q-1}|u_1,\dots,u_{q-2}]=0$. Therefore $\tilde{\lambda}_n=\lambda_n$.
\end{proof}

\section{Existence of equilibrium}
\label{section_existence}

In this section we aim to prove the existence of an equilibrium in $G_n(\mu)$. According to section \ref{section_reduced_eq}, we can focus on the game $\overline{G}_n(\mu)$. According to Corollary \ref{Cor:C1234} we just have to prove the existence of a pair $(\Pi_n,\Psi_n)$ such that conditions \textbf{(C1)} to \textbf{(C4)} are satisfied.

These conditions on $(\Pi_n,\Psi_n)$ lead to new conditions on $(\overline{\Pi}_n ,\Psi_n)$ where $\overline{\Pi}_n \in \Delta(\mathbb{R}^2)$ is the marginal  of $\Pi_n$ on $(L,S_n)$. As explained in the next subsection, there corresponds  an equilibrium  $(\Pi_n,\Psi_n)$ to a pair $(\overline{\Pi}_n ,\Psi_n)$ satisfying these new conditions.
We therefore will focus on these pairs $(\overline{\Pi}_n ,\Psi_n)$.

\subsection{The marginal $\overline{\Pi}_n$ }

$\Pi_n$ is a probability on $\Omega_n \times \mathbb{R}$ and it induces a marginal law $\overline{\Pi}_n \in \Delta(\mathbb{R}^2)$ for the pair $(S_n,L)$. \textbf{(C1)}, \textbf{(C2)} and \textbf{(C3)} are in fact conditions on $(\overline{\Pi}_n,\Psi_n)$. \textbf{(C4)} is the unique condition that involves the conditional law of $L$ given $\omega$. As proved with the first claim of the forthcoming Lemma \ref{lemma:C4prime}, it turns out that \textbf{(C4)} implies the following necessary condition on $\overline{\Pi}_n$ and $\Psi_n$:

\[ \textbf{(C4'):} ~~  \text{There exists a constant $\alpha_n$ such that }\frac{\partial\overline{\lambda}_n}{\partial\overline{\Pi}_{n | S}} = \alpha_n E_{\overline{\Pi}_n}[ H'(LS_n-\Psi_n(S_n)) | S_n  ]\]



It is useful to note that various equilibria $(\Pi_n,\Psi_n)$ could have the same marginal $\overline{\Pi}_n$. On the other hand, we will prove in Corollary \ref{from_nu_to_equilibrium} the existence of pairs $(\overline{\Pi}_n,\Psi_n)$ that satisfy \textbf{(C1)}, \textbf{(C2)}, \textbf{(C3)} and \textbf{(C4')}. To prove the existence of reduced equilibrium in $G_n(\mu)$ we therefore need the second claim of the next lemma:

\begin{lemma}\label{lemma:C4prime}~~\\
1/ Any reduced equilibrium $(\Pi_n,\Psi_n)$ in $G_n(\mu)$ is such that $(\overline{\Pi}_n,\Psi_n)$ satisfies \textbf{(C1)}, \textbf{(C2)}, \textbf{(C3)} and \textbf{(C4')}, where $\overline{\Pi}_n=\Pi_{n|(L,S_n)}$.\\ 2/ Conversely, to any $(\overline{\Pi}_n,\Psi_n)$ satisfying \textbf{(C1)}, \textbf{(C2)}, \textbf{(C3)} and \textbf{(C4')}, there corresponds at least one equilibrium $(\Pi_n,\Psi_n)$ such that $\Pi_{n|(L,S_n)}=\overline{\Pi}_n$.
\end{lemma}

\begin{proof}
We start with the first claim. We just have to prove that \textbf{(C4)} implies \textbf{(C4')}. Let $\Phi$ be a continuous and bounded function. According to \textbf{(C3)} we have:

\begin{align*}
E_{\lambda_n}[\Phi(S_n(\omega))]=&E_{\Pi_{n|\omega}}[\Phi(S_n(\omega))\frac{d\lambda_n}{d\Pi_{n|\omega}}]\\
=&E_{\Pi_n}[\Phi(S_n(\omega))\alpha_n E_{\Pi_n}[H'(LS_n-\Psi_n(S_n))|\omega ] ]\\
=&E_{\Pi_n}[\Phi(S_n(\omega))\alpha_n H'(LS_n-\Psi_n(S_n)) ]\\
=&E_{\Pi_n}[\Phi(S_n(\omega))\alpha_n E_{\Pi_n}[H'(LS_n-\Psi_n(S_n))|S_n ] ]\\
\end{align*}
Therefore $E_{\overline{\lambda}_n}[\Phi(S_n)]=E_{\lambda_n}[\Phi(S_n(\omega))]=E_{\overline{\Pi}_{n|S}}[\Phi(S_n(\omega))\alpha_n E_{\overline{\Pi}_n}[H'(LS_n-\Psi_n(S_n)|S_n] ]$ 
which is exactly our condition \textbf{(C4')}.

We now prove the second claim. Let $(\overline{\Pi}_n,\Psi_n)$ satisfy \textbf{(C1)}, \textbf{(C2)}, \textbf{(C3)} and \textbf{(C4')}. Consider then the probability $\Pi_n$ induced by the following lottery: select first $L$ and $S_n$ according to $\overline{\Pi}_n$. If $S_n=s$, select an history $\omega$ with the uniform probability on the set $K_s=\{ \omega | S_n(\omega)=s \}$. 

The marginal of $\Pi_n$ on $(L,S_n)$ coincides with $\overline{\Pi}_n$ and $(\Pi_n,\Psi_n)$ satisfies therefore \textbf{(C1)}, \textbf{(C2)} and \textbf{(C3)}.

Observe then that under $\Pi_n$, $L$ is then independent of $\omega$ given $S_n$ and therefore the conditional law of $(L,S_n)$ given $\omega$ coincides with  the conditional law of $(L,S_n)$ given $S_n$. So: $E_{\Pi_n}[H'(LS_n-\Psi_n(S_n))| \omega]=E_{\overline{\Pi}_n}[H'(LS_n-\Psi_n(S_n))| S_n]$, and \textbf{(C4)} then follows from \textbf{(C4')}.
\end{proof}

\subsection{Reformulation of \textbf{(C1)}, \textbf{(C2)} and \textbf{(C3)}}\label{subsection_reformulation}
In this subsection we show that a pair $(\overline{\Pi}_n,\Psi_n)$ satisfying \textbf{(C1)}, \textbf{(C2)} and \textbf{(C3)} is completely determined by the marginal law $\nu:=\overline{\Pi}_{n | S_n}$ of $S_n$.

It will be convenient to introduce the following notation: $\Delta(\mathbb{R}^2,\mu,\nu)$ is the set of probability distributions on $(L,S_n) \in \mathbb{R}^2$ with respective marginal laws $\mu$ and $\nu$.

\begin{definition}
For $\nu \in \Delta(\mathbb{R})$, we define $\phi_{\nu}(\ell):=F_{\nu}^{-1} (F_{\mu} (\ell))$ and $\gamma_{\nu}(s):=F_{\mu}^{-1} (F_{\nu} (s))$ where $F_{\mu}$ and $F_{\nu}$ are the cumulative distribution functions of $\mu$ and $\nu$, and $F_{\mu}^{-1}$  and $F_{\nu}^{-1}$ are their right inverses i.e. $F_{\nu}^{-1}(y)= \inf\{ x \mid F_{\nu}(x)> y\}$.

We further define:
\begin{equation}\label{def:Gamma}
\Gamma_{\nu}(s):=\int_0^s \gamma_{\nu}(t)dt
\end{equation}

\begin{equation}\label{def:Phi}
\Phi_{\nu}(\ell):=\int_0^{\ell} \phi_{\nu}(t)dt
\end{equation}   

We denote $\overline{\Pi}_{\nu}$  the law of the pair $(L,\phi_{\nu}(L))$ when $L$ is $\mu$-distributed.

Finally, for $\lambda \in \Delta(\mathbb{R})$, we set: \begin{equation}\label{def:Psi}
\Psi_{\nu,\lambda}(s):=\Gamma_{\nu}(s)-E_{\lambda}[\Gamma_{\nu}]
\end{equation}
\end{definition}


\begin{lemma}\label{lemma:nu}
Let $(\overline{\Pi},\Psi)$ be a pair where $\Psi$ is a convex function and where $\overline{\Pi} \in \Delta(\mathbb{R}^2)$  satisfies $\overline{\Pi}_{|S_n}=\nu$.  Then $(\overline{\Pi},\Psi)$ satisfies \textbf{(C1)}, \textbf{(C2)} and \textbf{(C3)} if and only if $(\overline{\Pi},\Psi)=(\overline{\Pi}_{\nu},\Psi_{\nu,\overline{\lambda}_n})$.
\end{lemma}

\begin{proof}
We first prove that the pair $(\overline{\Pi}_{\nu},\Psi_{\nu,\overline{\lambda}_n})$ satisfies \textbf{(C1)}, \textbf{(C2)} and \textbf{(C3)}. We start by observing that $\overline{\Pi}_{\nu} \in \Delta(\mathbb{R}^2,\mu,\nu)$. Indeed, according to the definition of $\overline{\Pi}_{\nu}$ the marginal law of $L$ is $\mu$. On the other hand, since $\mu$ has no atom, $U:=F_{\mu}(L)$ is uniformly distributed and it is well known that the Smirnov transform $F_{\nu}^{-1}(U)$ is $\nu$-distributed. Therefore the marginal law of $\overline{\Pi}_{\nu}$ on $S_n$ is just $\nu$. 

$\Psi_{\nu,\lambda}$ is a convex function since $\gamma_{\nu}$ is increasing and thus $\Gamma_{\nu}$ is convex. It further satisfies \textbf{(C1)} since, due to equation (\ref{def:Psi}), $E_{\overline{\lambda}_n}[\Psi_{\nu,\overline{\lambda}_n}]=0$. 

$\overline{\Pi}_{\nu}$ satisfies \textbf{(C2)} since it belongs to $\Delta(\mathbb{R}^2,\mu,\nu)$.

$\gamma_{\nu}$ is right continuous and therefore it follows from the definition of $\Gamma_{\nu}$ that $\partial \Psi_{\nu,\lambda}(s)=[\gamma_{\nu}(s^-),\gamma_{\nu}(s)]$ where $\gamma_{\nu}(s^-)$ is the left limit of $\gamma_{\nu}$ at $s$. Under $\overline{\Pi}_{\nu}$, $S_n=\phi_{\nu}(L)$. Therefore, condition \textbf{(C3)} is equivalent to:

\begin{equation}\label{eq:C4_equiv}
\overline{\Pi}_{\nu}\left[\gamma_{\nu}((\phi_{\nu}(L))^-)\leq L \leq \gamma_{\nu}(\phi_{\nu}(L))\right]=1
\end{equation}

We first prove that for all $x$:

\begin{equation}\label{eq:F_nu F_nu^-1}
F_{\nu}((F_{\nu}^{-1}(x))^-)       \leq x \leq F_{\nu}(F_{\nu}^{-1}(x))
\end{equation} 

Let $A:=\{ s | F_{\nu}(s)>x \}$ and $\alpha:=F_{\nu}^{-1}(x)$. It results from the definition of $F_{\nu}^{-1}(x)$ that $\alpha$ is the infimum of $A$. Furthermore, since $F_{\nu}$ is increasing, $]\alpha,\infty[ \subset A \subset [\alpha,\infty[$. Since $F_{\nu}$ is right continuous, we get $F_{\nu}(\alpha)=\displaystyle \inf_{s \in A} F_{\nu}(s)$. But if $s \in A$, $F_{\nu}(s)>x$, and therefore $F_{\nu}(\alpha)=\displaystyle \inf_{s \in A} F_{\nu}(s) \geq x$ and the right hand inequality in (\ref{eq:F_nu F_nu^-1}) is proved.

On the other hand, $F_{\nu}(\alpha^{-})=\lim_{u \rightarrow \alpha, u<\alpha}F_{\nu}(u)$. But if $u < \alpha$, $u \in A^c$ and thus $F_{\nu}(u) \leq x$. Therefore $F_{\nu}(\alpha^-)\leq x$ which is the second inequality. Replace $x$ by $F_{\mu}(L)$ in (\ref{eq:F_nu F_nu^-1}) to obtain: $F_{\nu}((\phi_{\nu}(L))^-)\leq F_{\mu}(L) \leq F_{\nu}(\phi_{\nu}(L))$. Since $F_{\mu}$ is increasing and one to one, we get therefore $F_{\mu}^{-1}(F_{\nu}((\phi_{\nu}(L))^-))\leq L \leq F_{\mu}^{-1}(F_{\nu}(\phi_{\nu}(L)))$ which is exactly (\ref{eq:C4_equiv}) according to the definition of $\gamma_{\nu}$, and $(\overline{\Pi}_{\nu},\Psi_{\nu,\lambda})$ satisfies thus \textbf{(C3)}.

~~\\
We now prove the converse statement. Let $\overline{\Pi}_n$ belong to $\Delta(\mathbb{R}^2,\mu,\nu)$ and $\Psi_n$ be a convex function such that $(\overline{\Pi}_n,\Psi_n)$ satisfies \textbf{(C1)}, \textbf{(C2)} and \textbf{(C3)}. We have to prove that $(\overline{\Pi},\Psi)=(\overline{\Pi}_{\nu},\Psi_{\nu,\overline{\lambda}_n})$

Being convex, the function $\Psi_n(s)$ has a derivative $\rho(s)$ at any point except maybe on a countable set. The function $\rho$ can be extended into a right continuous function defined for all $s \in \mathbb{R}$.
We then obtain that for all $s$, $\partial \Psi_n(s)=[\rho(s^-),\rho(s)]$. Since $\ell \in \partial\Psi_n(s) \Leftrightarrow s \in \partial\Psi_n^{\sharp}(\ell)$ we get, according to Fenchel equation (\ref{eq:fenchel}):  $\partial\Psi_n^{\sharp}(\ell)=[\rho^{-1}(\ell^-),\rho^{-1}(\ell)]$ where $\rho^{-1}(\ell):=\inf \{ s | \rho(s)> \ell \}$.

Condition \textbf{(C3)} implies therefore $\overline{\Pi}_n(\rho^{-1}(L^-)\leq S_n \leq \rho^{-1}(L))=1$. Observing that $\rho^{-1}$ is an increasing function, there are at most countably many points in $A:=\{ \ell | \rho^{-1}(\ell^-) \neq \rho^{-1}(\ell) \}$. Since $\mu$ is non atomic, $\mu(A)=0$ and thus $\overline{\Pi}_n[S_n=\rho^{-1}(L)]=1$. It follows that, under $\overline{\Pi}_n$, $(L,S_n)$ has the same law as $(L,\rho^{-1}(L))$. Since $\overline{\Pi}_n \in \Delta(\mathbb{R}^2,\mu,\nu)$, we conclude that $\rho^{-1}(L)$ is $\nu$-distributed when $L$ is $\mu$-distributed. As observed in the beginning of this proof $\phi_{\nu}(L)\sim \nu$ when $L \sim \mu$.  It turns out that $\phi_{\nu}$ is the unique right continuous increasing function having that property\footnote{Let indeed $f_1$, $f_2$ be two right continuous increasing functions such that $f_i(L) \sim \nu$ when $L \sim \mu$. Then for all $a \in \mathbb{R}$, $A_i:=\{ \ell | f_i(\ell) \geq a \}$ is a closed set. Since $f_i$ is increasing, $A_i$ must be an half line and we must have therefore $A_i=[\alpha_i,\infty[$. Since $f_i(L) \sim \nu$ and $F_{\mu}$ is continuous, we get:
\[ \nu([a,\infty[) = \mu(f_i(L) \geq a) = \mu(L \geq \alpha_i) = 1 -F_{\mu}(\alpha_i)
\] Therefore $F_{\mu}(\alpha_1)=F_{\mu}(\alpha_2)$ and thus $\alpha_1=\alpha_2$, since $F_{\mu}$ is strictly increasing according to the hypothesis \textbf{A1} on $\mu$. As a result, $A_1=A_2$, or in other words: for all $\ell$ and for all $a$, $f_1(\ell) \geq a$ if and only if $f_2(\ell)\geq a$. We conclude therefore that $f_1=f_2$.}, and we may therefore conclude that $\rho^{-1}=\phi_{\nu}$.

It follows on one hand that $\overline{\Pi}_n=\overline{\Pi}_{\nu}$. On the other hand, $\rho=\phi_{\nu}^{-1}=\gamma_{\nu}$. Therefore, $\partial\Psi_n(s)=\partial\Gamma_{\nu}(s)$ for all $s$. As a consequence $\Psi_n=\Gamma_{\nu}+c$ where $c$ is a constant. Since $\Psi_n$ satisfies (\textbf{C1}), we conclude that $c= -E_{\overline{\lambda}_n}[\Gamma_{\nu}]$ and thus $\Psi_n=\Psi_{\nu,\overline{\lambda}_n}$ as announced.\end{proof}

As explained in the introduction of this section, we are seeking for pairs $(\overline{\Pi}_n,\Psi_n)$ satisfying \textbf{(C1)}, \textbf{(C2)}, \textbf{(C3)} and \textbf{(C4')}. According to Lemma \ref{lemma:nu}, this is equivalent to find $\nu$ such that $(\overline{\Pi}_{\nu},\Psi_{\nu,\overline{\lambda}_n})$ satisfies \textbf{(C4')}.

 \textbf{(C4')} is a condition on the density of $\overline{\lambda}_n$ with respect to the marginal of $\Pi_{\nu | S_n}=\nu$. It expresses that this density $\frac{\partial \overline{\lambda}_n}{\partial \nu}$ is proportional to $Y_{\nu,\lambda}$ defined as:

\begin{equation}\label{def:Y_n}
Y_{\nu,\lambda}:= E_{\overline{\Pi}_{\nu}} [ H'(LS_n- \Psi_{\nu,\lambda}(S_n)) \mid S_n]
\end{equation}

Since $H'$ is strictly positive, so is $Y_{\nu,\lambda}$. Therefore we define:

\begin{equation}\label{def:alphanulambda}
\alpha_{\nu,\lambda}:=\frac{1}{ E_{\nu}[Y_{\nu,\lambda}]}
\end{equation}

as the unique constant $\alpha_{\nu,\lambda}$ such that $\alpha_{\nu,\lambda} . Y_{\nu,\lambda} . \nu$ is a probability measure (the notation $\alpha_{\nu,\lambda} . Y_{\nu,\lambda} . \nu$ refers to the measure $\zeta$ such that $\frac{\partial \zeta}{\partial \nu} = \alpha_{\nu,\lambda} . Y_{\nu,\lambda}$ ).
 
\begin{definition}\label{def:operator_T}

For $\lambda \in \Delta(\mathbb{R})$, $T_{\lambda}$ is defined as the map from $\nu \in \Delta(\mathbb{R})$ to $T_{\lambda}(\nu):=\alpha_{\nu,\lambda}. Y_{\nu,\lambda}. \nu \in \Delta(\mathbb{R})$, where
$\alpha_{\nu,\lambda}$ and $Y_{\nu,\lambda}$ are defined in equations (\ref{def:Y_n}) and (\ref{def:alphanulambda}).
\end{definition} 
 
With this definition, we get:
 
\begin{lemma}\label{lemma:C4'}
For all $\nu$, the pair $(\overline{\Pi}_{\nu},\Psi_{\nu,\overline{\lambda}_n})$ satisfies \textbf{(C4')} if and only if $T_{\overline{\lambda}_n}(\nu)=\overline{\lambda}_n$.
\end{lemma}
 
\begin{proof}
This results from the definition of $T_{\lambda}$ and the condition \textbf{(C4')}.
\end{proof}

The operator $T_{\lambda}$ is the central tool of our analysis. It is used both to prove the existence of equilibria and to prove their convergence.

In the next subsection, we analyze the continuity property of $T_{\lambda}$.


\subsection{Continuity of $T_{\lambda}$}


We first introduce the Wasserstein distance $W_2$ and we remind some of its useful properties. Our the continuity result for $T_{\lambda}$ is stated in Proposition \ref{cor:continuity}. Its technical proof is given in Annex \ref{continuity_annexe}.


\begin{definition}
For $d \in \mathbb{N}^*$, we define $P_2(\mathbb{R}^d)$ the Wasserstein space of order $2$ on $\mathbb{R}^d$, as:

\[ P_2(\mathbb{R}^d):= \{  \nu \in  \Delta( \mathbb{R}^d), \text{ such that } \int_{\mathbb{R}^d} \| x \| ^ 2 \nu(dx) < \infty   \} \]

For $\nu_1,\nu_2 \in P_2(\mathbb{R}^d)$ we define the Wasserstein distance between $\nu_1$ and $\nu_2$ as:

\[ W_2(\nu_1,\nu_2)= \left(  \underset{\pi \in \Delta(\mathbb{R}^{2d},\nu_1,\nu_2)}{\inf} \int_{\mathbb{R}^{2d}}  \| x - y \| ^ 2 d\pi(x,y)   \right)^{\frac{1}{2}}
\]
\end{definition}

$W_2$ is clearly finite on $P_2(\mathbb{R}^d)$ and $(P_2(\mathbb{R}^d),W_2)$ is a metric space. This metric is useful to deal with weak convergences (as indicates Proposition \ref{prop:W_2_for_weak}).\\

 Remember that a sequence $\nu_k \in \Delta(\mathbb{R}^d)$ is said to converge in law, or to convergence weakly in $\Delta(\mathbb{R}^d)$ to $\nu$ if and only if for any bounded continuous function $\phi: \mathbb{R}^d\rightarrow \mathbb{R}$, we have $E_{\nu_k} [\phi] \rightarrow E_{\nu} [\phi]$ as $k \rightarrow \infty$.

There exists also a weak convergence in $P_2(\mathbb{R}^d)$: $\nu_k$ converges to $\nu$ weakly in $P_2(\mathbb{R}^d)$ if and only if for any continuous functions $\phi$ satisfying for some constant $C \in \mathbb{R}$: $ \forall x \in \mathbb{R}^d$, $\mid \phi(x) \mid \leq C(1+\|x\|)^2$, we have that $E_{\nu_k} [\phi] \rightarrow E_{\nu} [\phi]$ as $k \rightarrow \infty$.

The following proposition is well known (see for instance theorem 6.9 in \citet{villani}, or \citet{mallows} for a proof). It makes the link between weak convergences and $W_2$ convergence.

\begin{proposition}\label{prop:W_2_for_weak}
The three following statements are equivalent:

1/ $W_2(\nu_n,\nu) \rightarrow 0$

2/ $\nu_n \rightarrow \nu$ (weakly in $P_2(\mathbb{R}^d))$

3/ $\nu_n \rightarrow \nu$ (weakly in $\Delta(\mathbb{R}^d))$ and $E_{\nu_n}[\|s\|^2] \rightarrow E_{\nu}[\|s\|^2]$.
\end{proposition}

The continuity result is expressed in the next proposition. 

\begin{proposition}\label{cor:continuity}
If $\nu_k$ and $\lambda_k$ are two sequences of measure in $P_2(\mathbb{R})$ that converge in $W_2$ distance to $\nu$ and $\lambda$, then $W_2(T_{\lambda_k}(\nu_k),T_{\lambda}(\nu))\rightarrow 0$ 
\end{proposition}

The proof of this proposition is postponed to the Annex \ref{continuity_annexe}.\\

\subsection{KKM theorem and existence of equilibrium.}
We are now ready to state the existence of a Nash equilibrium in $\overline{G}_n(\mu)$. According to Lemma \ref{lemma:C4prime} and \ref{lemma:C4'}, to prove the existence of an equilibrium in $\overline{G}_n(\mu)$, we have to show that there exists $\nu_n \in \Delta(\mathbb{R})$ such that $T_{\overline{\lambda}_n}(\nu_n)=\overline{\lambda}_n$. Remember that $\overline{\lambda}_n \in \Delta_f(\mathbb{R})$ where $\Delta_f(\mathbb{R})$ is the set of probability measures on $\mathbb{R}$ with finite support. Observe next that $T_{\lambda}(\nu)$ is defined by a density function with respect to $\nu$. In particular $T_{\lambda}(\nu)$ is absolutely continuous with respect to $\nu$, which is denoted $T_{\lambda}(\nu) \ll \nu $ (i.e. for all measurable set $A$, $\nu(A)=0 \Rightarrow T_{\lambda}(\nu)(A)=0$). Therefore $T_{\lambda}(\nu) \in \Delta_f(\mathbb{R})$ if $\nu \in \Delta_f(\mathbb{R})$.\\

The next theorem can then be applied to $T_{\overline{\lambda}_n}$ to conclude the existence of equilibrium.

\begin{theorem}
A map $T: \Delta_f(\mathbb{R}) \rightarrow \Delta_f(\mathbb{R})$ that is continuous for the $W_2$ metric and satisfies $T(\nu) \ll \nu$ for all $\nu$ is necessarily onto.
\end{theorem}

\begin{proof}
Let $\lambda$ be a measure in $\Delta_f(\mathbb{R})$ and denote $K$ its support. If $T(\nu) \ll \nu$, then necessarily the support of $T(\nu)$ is included in the support of $\nu$. Therefore $T$ maps $\Delta(K)$ to $\Delta(K)$.\\ $\Delta(K)$ can be identified with the $|K|$-dimensional simplex hereafter denoted $\Delta$ and the restriction of $T$ to $\Delta$ is a continuous map. It further preserves the faces $F_i:=\{ x \in \Delta | x_i=0 \}$. It follows from an argument used in \citet{gale} that $T$ is onto. Indeed, let $\lambda \in \Delta$ and define $C_i:= \{ x \in \Delta | T(x)_i \leq \lambda_i \}$. Since $T$ is continuous, $C_i$ is clearly a closed subset of $\Delta$. Furthermore, if $x \in F_i$ then $x_i=0$ and thus $T(x)_i=0 \leq \lambda_i$. We conclude therefore that for all $i$, $F_i \subset C_i$. We next argue that $\Delta \subset \cup_i C_i$. Indeed, for all $x \in \Delta$, $T(x) \in \Delta$. There must exists $i$ such that $T(x)_i \leq \lambda_i$. Otherwise we would have for all $i$, $T(x)_i > \lambda_i$, and summing all those inequalities we would get $1>1$. Therefore there exists $i$ such that $x \in C_i$. As announced, $\Delta \subset \cup_i C_i$.\\
 According to KKM theorem (see the particular version presented in \cite{mertens2014repeated} \begin{color}{red} page \end{color}) there exists $x \in \cap_i C_i$. So for this $x$ we get for all $i$ that $T(x)_i \leq \lambda_i$. Since the sum over $i$ of both sides equal to $1$, we infer that these inequalities are in fact equalities, and thus $T(x)=\lambda$. 
\end{proof}

Since our map $T_{\overline{\lambda_n}}$ is onto, we conclude that for all $n$, there exists $\nu_n$ such that $T_{\overline{\lambda}_n}(\nu_n)=\overline{\lambda}_n$. The corresponding pair $(\Pi_{\nu_n},\Psi_{\nu_n,\overline{\lambda}_n})$ satisfies \textbf{(C1)}, \textbf{(C2)}, \textbf{(C3)} and \textbf{(C4')}. We conclude then with Lemma \ref{lemma:C4prime}-2, that:

\begin{corollary}\label{from_nu_to_equilibrium}
There exists a reduced equilibrium in $G_n(\mu)$.
\end{corollary}

\section{Convergence of $\nu_n$}\label{section_convergence}

In order to describe the asymptotics of the price process, we have first to analyze the asymptotics of any sequence $(\nu_n)$ satisfying for all $n$ the equation $T_{\overline{\lambda}_n}(\nu_n)=\overline{\lambda}_n$. For now on, $(\nu_n)$ will denote any such sequence.\\

First observe that $\overline{\lambda}_n$ is the law of $S_n=\frac{\sum_{i=1}^n u_i}{\sqrt{n}}$ when $(u_1,\dots,u_n)$ are independent and centred. It follows from the central limit theorem that $\overline{\lambda}_n$ converges in law to $\overline{\lambda}_{\infty}:=\mathcal{N}(0,1)$. Observing that the second order moments $E_{\overline{\lambda}_{n}}[S_n^2]=1$ for all $n$, this weak convergence in $\Delta(\mathbb{R})$ implies (see Proposition \ref{prop:W_2_for_weak}) the $W_2$-convergence of $\overline{\lambda}_n$ to $\overline{\lambda}_{\infty}$.

We first prove in Lemma \ref{lemma:accumulation_point} that the sequence $(\nu_n)$ is relatively compact. As a consequence we infer with Corollary \ref{cor:accumulation_point} that any sequence $(\nu_n)$ must have an accumulation point satisfying 
\begin{equation}\label{eq:T_lam}T_{\overline{\lambda}_{\infty}}(\nu)=\overline{\lambda}_{\infty}
\end{equation}

 It turns out that equation (\ref{eq:T_lam}) for $\nu$ is equivalent to the claim that $\Psi_{\nu,\overline{\lambda}_{\infty}}$ is a smooth solution to a differential system (see Proposition \ref{prop:equiv}).\\
 
We next claim in Theorem \ref{Thm_uniqueness} that this differential problem $\mathcal{D}$ has a unique solution and therefore the equation $T_{\overline{\lambda}_{\infty}}(\nu)=\overline{\lambda}_{\infty}$ has a unique solution $\nu$ (see Corollary \ref{unique_nu}). This implies the convergence of $\nu_n$ to this unique solution $\nu$, as stated in Corollary \ref{cor:convergence_nu_n}.


 
 

\begin{lemma}\label{lemma:accumulation_point}
The sequence $(\nu_n)$ is relatively compact: any subsequence of $(\nu_n)$ has an accumulation point in $P_2(\mathbb{R})$.
\end{lemma}

\begin{proof}

We first prove that $E_{\nu_n}[S_n^2]$ is bounded. It follows immediately from the assumptions \textbf{A2} on $H$ as well as from the definition of $Y_{\nu,\lambda}$ and $\alpha_{\nu,\lambda}$ (see equation (\ref{def:Y_n})), that $\epsilon < Y_{\nu,\lambda} < K$, and $\frac{1}{K}<\alpha_{\nu,\lambda}<\frac{1}{\epsilon}$. Therefore: $\frac{\epsilon}{K}<\alpha_{\nu,\lambda}.Y_{\nu,\lambda}<\frac{K}{\epsilon}$. According to the definition of $\overline{\lambda}_{n}$, we have $E_{\overline{\lambda}_{n}}[S_n^2]=1$. And thus:

\[ 1=E_{\overline{\lambda}_{n}}[S_n^2]=E_{T_{\overline{\lambda}_{n}}(\nu_n)}[S_n^2]=E_{\nu_n}[\alpha_{\nu_n,\overline{\lambda}_{n}}Y_{\nu_n,\overline{\lambda}_{n}}S_n^2] \geq \frac{\epsilon}{K}E_{\nu_n}[S_n^2]
\]
which leads to $E_{\nu_n}[S_n^2] \leq \frac{K}{\epsilon}$. We conclude with Markov-Tchebichev inequality that for all $\eta>0$, $\nu_n(S_n^2 \geq \frac{\eta \epsilon}{K})\leq \eta $. This indicates that the sequence of measures $(\nu_n)$ is tight: for all $\eta>0$ there exists a compact set $[-\sqrt{\frac{M}{\eta}},\sqrt{\frac{M}{\eta}}]$ such that for all $n$, $\nu_n([-\sqrt{\frac{M}{\eta}},\sqrt{\frac{M}{\eta}}])\geq 1-\eta$. 

This tightness property implies with Prokhorov's theorem that there exists a subsequence $\nu_{n(k)}$ of $\nu_n$ that weakly converges to some $\nu \in \Delta(\mathbb{R})$. Since the second order moment are bounded, we may select a subsequence of $\nu_{n(k)}$ such that the second order moments converge.  
 According to Proposition \ref{prop:W_2_for_weak}, this implies the $W_2$-convergence of $\nu_{n(k)}$.\end{proof}

\begin{corollary}\label{cor:accumulation_point}
Any accumulation point $\nu$ of the sequence $(\nu_n)$ satisfies $T_{\overline{\lambda}_{\infty}}(\nu)=\overline{\lambda}_{\infty}$ where $\overline{\lambda}_{\infty}=\mathcal{N}(0,1)$.
\end{corollary}

\begin{proof}
Take a subsequence $\nu_{n(k)}$ converging to $\nu$ in $W_2$. Since we also have $\overline{\lambda}_{n(k)} \rightarrow \overline{\lambda}_{\infty}$ in $W_2$, we may apply our continuity result on $T$ (see Proposition \ref{cor:continuity}) to conclude $T_{\overline{\lambda}_{\infty}}(\nu)=\overline{\lambda}_{\infty}$.
\end{proof}


\begin{proposition}\label{prop:equiv}
Suppose that $\nu$ is a probability measure such that $T_{\overline{\lambda}_{\infty}}(\nu)=\overline{\lambda}_{\infty}$ with $\overline{\lambda}_{\infty}=\mathcal{N}(0,1)$, then:

1/ The function $\Psi_{\nu,\overline{\lambda}_{\infty}}$ (see Definition \ref{def:Phi}) is $C^2$.

2/ The pair $(\psi,c):=(\Psi_{\nu,\overline{\lambda}_{\infty}},\frac{1}{\alpha_{\nu,\overline{\lambda}_{\infty}}})$ is a solution of the following differential system $\mathcal{D}$:

\begin{equation*}
  (\mathcal{D}) \left\{
      \begin{array}{ll}
      (1)
     
&\forall s \in \mathbb{R}, f_{\mu}(\psi'(s))\psi''(s)H'(s\psi'(s)-\psi(s))=c\mathcal{N}(s) \\ (2)
&lim_{s \rightarrow -\infty} \psi'(s)=0\\ (3)
&lim_{s \rightarrow +\infty} \psi'(s)=1\\ (4)
&\int_{-\infty}^{+\infty}  \psi(z)\mathcal{N}(z)dz=0
      \end{array}
    \right.
    \end{equation*}

where $\mathcal{N}(z):=\frac{e^{-\frac{z^2}{2}}}{\sqrt{2\pi}}$.

\end{proposition}
This proposition is proved in Annex \ref{Annexe:equivalence}.

\begin{theorem}\label{Thm_uniqueness}
There exists at most one pair $(\psi,c)$ solution to the system $\mathcal{D}$.
\end{theorem}
This Theorem is proved in Annex \ref{annexe:unique_solution}.

\begin{corollary}\label{unique_nu}
There exists a unique measure $\nu$ such that $T_{\overline{\lambda}_{\infty}}(\nu)=\overline{\lambda}_{\infty}$ where $\overline{\lambda}_{\infty}=\mathcal{N}(0,1)$
\end{corollary}

\begin{proof}
If $\nu_1$ and $\nu_2$ are two solutions of $T_{\overline{\lambda}_{\infty}}(\nu)=\overline{\lambda}_{\infty}$, then the pairs $(\Psi_{\nu_i,\overline{\lambda}_{\infty}},\frac{1}{\alpha_{\nu_i,\overline{\lambda}_{\infty}}})$ for $i=1,2$ would be solutions of the system $\mathcal{D}$ according to  Proposition \ref{prop:equiv}. As a result of Theorem \ref{Thm_uniqueness}: $\Psi_{\nu_1,\overline{\lambda}_{\infty}}=\Psi_{\nu_2,\overline{\lambda}_{\infty}}$. Thus the derivatives of these functions also coincide: $\gamma_{\nu_1}=\gamma_{\nu_2}$ where $\gamma_{\nu_i}$ are defined in Definition \ref{def:Gamma}. Since $F_{\mu}$ is one-to-one, this implies that $F_{\nu_1}=F_{\nu_2}$ and thus $\nu_1=\nu_2$. 
\end{proof}



We are now ready to prove the main result of this section:


\begin{corollary}\label{cor:convergence_nu_n}
The sequence $(\nu_n)$ converges to the unique solution $\nu$ of $T_{\overline{\lambda}_{\infty}}(\nu)=\overline{\lambda}_{\infty}$.
\end{corollary}

\begin{proof}

Otherwise there would exists a subsequence $(\nu_{n(k)})$ that would not admit $\nu$ has accumulation point. This is impossible since this sequence would have an accumulation point $\tilde{\nu}$ according to Lemma \ref{lemma:accumulation_point} which should satisfy: $T_{\overline{\lambda}_{\infty}}(\tilde{\nu})=\overline{\lambda}_{\infty}$. According to Lemma \ref{unique_nu}, we would then have a contradiction: $\tilde{\nu}=\nu$. \end{proof}

\section{Convergence of the price process to a CMMV}
\label{section_CMMV}

Our analysis in this section applies to any sequence $(\Pi_n,X_n)$ of reduced equilibria in $G_n(\mu)$. We will focus on the  price process $(p_q^n)_{q=1,\dots,n}$ posted by player 2 in these equilibria. In a reduced equilibrium, the strategy $(p_q^n)_{q=1,\dots,n}$ of player 2 is pure (non random) but his moves depend on the past actions $\omega=(u_1,\dots,u_n)$ of player 1 which are random. The process $(p_q^n)_{q=1,\dots,n}$ is then a stochastic process. Its law when $\omega$ is $\Pi_{n | \omega}$-distributed is called the \textit{historical law}. On the other hand, the law of the price process when $\omega$ is $\lambda_n$-distributed is called the \textit{martingale equivalent law}.\\

We have seen in Theorem \ref{price_is_mart} that when $\omega \sim \lambda_n$, the process $(p_q^n)_{q=1,\dots,n}$ is a martingale. Furthermore $\lambda_n$ is the unique probability equivalent to $\Pi_{n | \omega}$ that has this property as stated in Theorem \ref{thm:unique_equiv_measure}.\\

Our purpose on this section is to analyze the asymptotics of the distribution of the prices process. In subsection \ref{subsection:convergence_Qn}, we analyze the limit of the martingale equivalent laws. In subsection \ref{subsection:convergence_Pn}, we analyze the asymptotics of the historical laws.

\subsection{Convergence of the martingale equivalent law.}\label{subsection:convergence_Qn}

Let $(\Pi_n,X_n)$ be a sequence of reduced equilibria in $G_n(\mu)$. We already know that $X_n=\Psi_{\nu_n,\overline{\lambda}_n}(S_n)$ and that $\overline{\Pi}_n=\Pi_{\nu_n}$ for a measure $\nu_n$ satisfying $T_{\overline{\lambda}_n}(\nu_n)=\overline{\lambda}_n$. According to formula (\ref{eq:price}), the price posted a period $q$ is:

\begin{equation}\label{equation:price}
p_q^n=\sqrt{n} E_{\lambda_n} [ u_n \Psi_{\nu_n,\overline{\lambda}_n}(S_n)  \mid u_1, \dots, u_{q-1} ]
\end{equation}

It is convenient to introduce the process  $Z^n: t \in ]0,1] \rightarrow Z_t^n := p^n_{\lfloor nt \rfloor}$ where $\lfloor x \rfloor$ is the largest integer less or equal to $x$. This is a continuous time process that jumps to the next value of $p_q^n$ at time $t=\frac{q}{n}$. We  analyze in this section the asymptotics of the law $Q_n$ of $Z^n$ when $(u_1,\dots,u_n)$ are endowed with the probability $\lambda_n$.

Let us introduce the notation $S^n_q=\frac{\sum_{i=1}^q u_i}{\sqrt{n}}$. Formula (\ref{equation:price}) can be written as:

\begin{align}
p_q^n=&\sqrt{n} E_{\lambda_n} [ u_n \Psi_{\nu_n,\overline{\lambda}_n}(S_{n-1}^n+\frac{u_n}{\sqrt{n}})  \mid u_1, \dots, u_{q-1} ]\nonumber \\
=&\frac{\sqrt{n}}{2} E_{\lambda_n} [  \Psi_{\nu_n,\overline{\lambda_n}}(S_{n-1}^n+\frac{1}{\sqrt{n}})-\Psi_{\nu_n,\overline{\lambda}_n}(S_{n-1}^n-\frac{1}{\sqrt{n}})  \mid u_1, \dots, u_{q-1} ] \label{equation:price_S_n} \\
=&\frac{\sqrt{n}}{2} E_{\lambda_n} [  \Psi_{\nu_n,\overline{\lambda}_n}(S_{n-1}^n+\frac{1}{\sqrt{n}})-\Psi_{\nu_n,\overline{\lambda}_n}(S_{n-1}^n-\frac{1}{\sqrt{n}})  \mid S_{q-1}^n ] \nonumber
\end{align}

Heuristically we have that $p_q^n \approx E_{\lambda_n} [  \Psi_{\nu_n,\overline{\lambda}_n}^{\prime}(S_{n-1}^n)  \mid S_{q-1}^n ]$.
From Corollary \ref{cor:convergence_nu_n}, we have that $\nu_n$ converges to $\nu$. Furthermore, according to Donkster theorem, $S^n_{\lfloor tn \rfloor }$ converges in law to $B_t$ where $B$ is a standard Brownian motion. We can heuristically expect therefore that $Z^n_t$ converges in law to $Z_t:=E[\Psi'_{\nu,\overline{\lambda}_{\infty}}(B_1)|B_t]$. This will be our focus in this section.

Observe next that $Z$ must be a CMMV. Indeed $\Psi'$ is an increasing function and we may apply the following lemma.

\begin{lemma}\label{le:X_is_a_CMMV}
If $B$ is a Brownian motion on a filtration $(\mathcal{F}_t)$ and if $g$ is an increasing function $\mathbb{R} \rightarrow \mathbb{R}$, then $X_t:=E[g(B_1) | \mathcal{F}_t]$ is a CMMV. \end{lemma}

\begin{proof}
Due to the Markov property of the  Brownian motion, we have $X_t=E[g(B_1) | \mathcal{F}_t]=E[g(B_1) | B_t]=E[g(B_t + (B_1-B_t)) | B_t]$. We get therefore $X_t= f(B_t,t)$, where $f(x,t)=E_z[g(x+z\sqrt{1-t})]$, with $z \sim \mathcal{N}(0,1)$. Note that $f$ is the convolution of $g$ with a normal density kernel. This convolution preserves the class of increasing functions and $f$ is thus indeed increasing in $x$. It is further $C^2$ due to the smoothing property of the normal kernel.
\end{proof}

Let us now prove formally the convergence of $Z_t^n$ to $Z_t$. We start by reminding the definition of the weak convergence in finite distributions of a sequence of stochastic processes:

\begin{definition}\label{definition_process_convergence}
A sequence $(Z^n)$ of processes converges in finite dimensional distribution to a process $Z$ if and only if for all finite family $J$ of times $(t_1<\dots<t_k)$, the random vectors $(Z^n_t)_{t \in J}$ converge in law to the random vector $(Z_t)_{t \in J}$.
\end{definition}

Our main theorem is then:

\begin{theorem}\label{thm:equiv}
Under the equivalent martingale measure, $(Z^n)$ converges in finite dimensional distribution to the CMMV $Z$ where $Z_t:= E[\Psi_{\nu,\overline{\lambda}_{\infty}}'(B_1)|B_t]$
\end{theorem}

\begin{proof}
We will prove this convergence by proving that the $W_2(\rho_n,\rho) \rightarrow 0$ when $\rho_n$ and $\rho$ are respectively the laws of the vectors $(Z^n_t)_{t \in J}$ and $(Z_t)_{t \in J}$. We use "Skorokhod representation" techniques to get that result. Let $(\tilde{\Omega},\mathcal{A},\tilde{P})$ be a probability space on which $B$ is a Brownian motion. In this section, unless otherwise stated, all expectations on $\tilde{\Omega}$ are taken with respect to $\tilde{P}$. $Z_t=E[\Psi'_{\nu,\overline{\lambda}_{\infty}}(B_1)|B_t]$ can be considered as a process on that space. \\

We will introduce hereafter a sequence of processes $\tilde{Z}^n$ defined on $\tilde{\Omega}$ such that:
\begin{equation}\label{eq:2conditions}
\begin{array}{ll}
1/ & \mbox{$\tilde{Z^n}$ and $Z^n$ have the same laws.}\\
2/ & \mbox{$\displaystyle \sup_{t} \|\tilde{Z}^n_t-Z_t \|_{L^2} \rightarrow 0$.} \\
\end{array}
\end{equation}

Theorem will then be proved. Indeed, $(\tilde{Z}^n,Z)$ is a pair of processes on the same probability space $(\tilde{\Omega},\mathcal{A},\tilde{P})$. The joint joint law of $(\tilde{Z}^n_t,Z_t)_{t \in J}$ is a probability distribution on $\mathbb{R}^{2|J|}$ with respective marginals $\rho_n$ and $\rho$. Therefore:

$$W_2(\rho_n,\rho) \leq \sqrt{E \left[\sum_{t \in J} |\tilde{Z}^n_t-Z_t|^2\right]}=\sqrt{ \left[\sum_{t \in J} \|\tilde{Z}^n_t-Z_t\|_{L^2}^2\right]}  \leq \sqrt{|J|}\displaystyle \sup_{t \in ]0,1]} \|\tilde{Z}^n_t-Z_t \|_{L^2} \rightarrow 0$$

In order to construct those random variables $\tilde{Z}^n$, it is convenient to apply the embedding techniques already used in \citet{de2010}. Let $\mathcal{F}_t$ denote the natural filtration of the Brownian motion $B$. Define $\tau^n_0=0$ and, recursively, $\tau^n_{q+1}$ as the first time $t>\tau^n_q$ such that $| B_t-B_{\tau^n_q} |=\frac{1}{\sqrt{n}}$. Since the one-dimensional Brownian motion is a recurrent process $\tau^n_{q}<\infty$ almost surely and clearly $\tilde{u}_q := \sqrt{n} (B_{\tau^n_{q}}-B_{\tau^n_{q-1}})$ has the same distribution as $u_q$ under $\lambda_n$. Indeed $\tilde{u}_q \in \{ -1,+1 \}$ and $E[\tilde{u}_q]=0$. They are furthermore independent since the increments $B_{\tau^n_{q}}-B_{\tau^n_{q-1}}$ are independent of $\mathcal{F}_{\tau^n_{q-1}}$.

Therefore, $B_{\tau^n_q}=\sum_{j=1}^{q} (B_{\tau^n_{j}}-B_{\tau^n_{j-1}})=\frac{1}{\sqrt{n}}\sum_{j=1}^{q} \tilde{u}_j$ has the same distribution as $S^n_q$ under $\lambda_n$. We set:

\begin{equation}\label{eq:def_z_n^n}
\tilde{z}_n^n:=\frac{\sqrt{n}}{2}\left(\Psi_n(B_{\tau^n_{n-1}} + \frac{1}{\sqrt{n}})-\Psi_n(B_{\tau^n_{n-1}} - \frac{1}{\sqrt{n}})\right)
\end{equation}

$\tilde{z}_n$ has then the same distribution as $p_{n}^n$. Furthermore, if we define:

$$\tilde{z}^n_q:= E[\tilde{z}_n|\tilde{u}_1,\dots,\tilde{u}_{q-1}]=E[\tilde{z}_n| \mathcal{F}_{\tau^n_{q-1}}]$$ the process $(\tilde{z}^n_q)_{q=1,\dots,n}$ has the same distribution as the process $(p^n_q)_{q=1,\dots,n}$ under $\lambda_n$, as it follows from equations (\ref{equation:price_S_n}) and (\ref{eq:def_z_n^n}). We next define:

\[ \tilde{Z}^n_t:=\tilde{z}^n_{\lfloor nt \rfloor}
\]

It is then clear that $\tilde{Z}^n$ and $Z^n$ have the same laws which claim 1 in (\ref{eq:2conditions}). We next prove claim 2:
\begin{align*}
\| \tilde{Z}^n_t - Z_t \|_{L^2}=&\|  E[\tilde{z}_n^n| \mathcal{F}_{\tau^n_{\lfloor nt \rfloor-1}}]-Z_t \|_{L^2}\\
 \leq & \|  E[\tilde{z}_n^n| \mathcal{F}_{\tau^n_{\lfloor nt \rfloor-1}}]-Z_{\tau^n_{\lfloor nt \rfloor-1}} \|_{L^2}+
\|Z_{\tau^n_{\lfloor nt \rfloor-1}}-Z_t \|_{L^2}\\
=&\|  E[\tilde{z}_n^n-Z_1| \mathcal{F}_{\tau^n_{\lfloor nt \rfloor-1}}]\|_{L^2}+
\|Z_{\tau^n_{\lfloor nt \rfloor-1}}-Z_t \|_{L^2}\\
\leq& \|\tilde{z}_n^n-Z_1]\|_{L^2}+
\|Z_{\tau^n_{\lfloor nt \rfloor-1}}-Z_t \|_{L^2}
\end{align*}  

We next argue that both terms of the right hand side go to zero as $n$ goes to $\infty$.

Let us start with the second one. First observe that all the martingales on the Brownian filtration are continuous (see \citet{revuz_yor}, theorem V.3.5), and $Z_t= E[Z_1|\mathcal{F}_t]$ in particular.  
If $\|Z_{\tau^n_{\lfloor nt \rfloor-1}}-Z_t \|_{L^2}$ was not converging to zero, there would exists a subsequence $n(k)$ such that 
$(Z_{\tau^{n(k)}_{\lfloor n(k)t \rfloor-1}})$ does not admit $Z_t$ as accumulation point in $L^2$. We prove in Lemma \ref{lemma:tau} that $\tau^n_{\lfloor nt \rfloor -1} \rightarrow  t$ in $L^2$. The sequence $n(k)$ can thus be selected such that $\tau^{n(k)}_{\lfloor n(k)t \rfloor -1} \rightarrow  t$ almost surely. By continuity we get then that $(Z_{\tau^{n(k)}_{\lfloor n(k)t \rfloor-1}})$ converges almost surely to $Z_1$ and the convergence also holds in $L^2$ since $(Z_t)$ is uniformly integrable ($Z_1= \Psi_{\nu,\overline{\lambda}_{\infty}}'(B_1)=F_{\mu}^1(F_{\nu}(B_1))$ is bounded). This contradicts the definition of the subsequence $n(k)$.

Assume now that the first term does not converge to zero. There would exist a subsequence $n(k)$ such that $\tilde{z}_{n(k)}^{n(k)}$ does not have $Z_1$ as accumulation point in $L^2$. 

Setting $a_{n(k)}:=B_{\tau^{n(k)}_{{n(k)}-1}} - \frac{1}{\sqrt{{n(k)}}}$ and $b_{n(k)}:=B_{\tau^{n(k)}_{{n(k)}-1}} + \frac{1}{\sqrt{{n(k)}}}$, equation (\ref{eq:def_z_n^n}) becomes $\tilde{z}_{n(k)}^{n(k)}=\frac{\Psi_{n(k)}(b_{n(k)})-\Psi_{n(k)}(a_{n(k)})}{b_{n(k)}-a_{n(k)}}$. With the mean value theorem, we conclude that there exists $x_{n(k)} \in [a_{n(k)},b_{n(k)}]$ such that $\tilde{z}_{n(k)}^{n(k)} \in \partial \Psi_{n(k)}(x_{n(k)})$.

But it follows from Lemma \ref{lemma:tau} here below that $B_{\tau^{n(k)}_{n(k)-1}}$ converges in $L^2$ to $B_1$. The subsequence $n(k)$ can thus be selected in such a way that $B_{\tau^{n(k)}_{{n(k)}-1}}$ converges to $B_1$ almost surely and so does $x_{n(k)}$. Since $\Psi_n=\Psi_{\nu_n,\overline{\lambda}_{n}}$  converges uniformly to $\Psi_{\nu,\overline{\lambda}_{\infty}}$ which is $C^2$, we may apply the forecoming Lemma \ref{prelim_lemma} to conclude that $\tilde{z}_n^n$ converges almost surely to $ \Psi_n'(B_1)=Z_1$. Since $\tilde{z}_n^n$ belongs to $\partial \Psi_n(k)(x_{n(k)}) \subset [0,1]$, it follows from the Lebesgue dominated convergence theorem that $\tilde{z}_{n(k)}^{n(k)}$ converges to $Z_1$ in $ L^2$, in contradiction with the definition of the subsequence $n(k)$. Hence, as announced both terms go to zero. Therefore both claims in (\ref{eq:2conditions}) are satisfied by the process $\tilde{Z}^n$ and Theorem \ref{thm:equiv} is thus proved. \end{proof}

We next prove the announced lemma.

\begin{lemma}\label{lemma:tau}~~\\
Claim 1: $\tau^n_{\lfloor nt \rfloor} \underset{L^2}{\longrightarrow}  t$\\ Claim 2: $B_{\tau^n_{n-1}} \underset{L^2}{\longrightarrow} B_1$
\end{lemma}

\begin{proof}
As well known:
 \[ E(\tau_q^n)=E(B_{\tau_q^n}^2)=E_{\lambda_n}( (S^n_q)^2  )= \frac{q}{n} \]

On the other hand, $\tau^n_{q+1}-\tau^n_{q}$ is independent of $\mathcal{F}_{\tau^n_q}$. Therefore, $\tau^n_q= \sum_{i=0}^{q-1} \tau^n_{i+1}-\tau^n_{i}$ is a sum of independent random variables with expectation $\frac{1}{n}$. 

Moreover we have $Var(\tau^n_{\lfloor nt \rfloor}) \rightarrow 0$ when $n \rightarrow \infty$.

Indeed:
\[  Var(\tau^n_{q+1}-\tau^n_{q}) \leq E((\tau^n_{q+1}-\tau^n_q)^2) \leq C E [ | B_{\tau^n_{q+1} }-B_{\tau^n_q} |^4 ]= C (\frac{1}{\sqrt{n}})^4     =\frac{C}{n^2}  \]

where C is the Burkholder's constant for $p=4$ (see Theorem IV.4.1 in \citet{revuz_yor}).

Therefore:

\[ Var(\tau^n_q) = \sum_{i=0}^{q-1} Var(\tau^n_{i+1}-\tau^n_{i}) \leq  \frac{q C}{n^2}  \leq  \frac{C}{n} \]

And:
\begin{equation}\label{eq:maj_burkh}
 \| \tau_q^n - \frac{q}{n}  \|_{L^2} ^2 = \| \tau_q^n - E[\tau^n_q]  \|_{L^2} ^2 = Var(\tau^n_q) \leq \frac{C}{n}
\end{equation}

Replacing $q$ by $\lfloor nt \rfloor$, we get claim 1 as announced.

It is also well known that $E[(B_{\tau^n_{n-1}}- B_1)^2]=
E[|\tau^n_{n-1}-1|]$. With equation (\ref{eq:maj_burkh}) we get:

$\|B_{\tau^n_{n-1}}- B_1\|_{L^2}^2=
\|\tau^n_{n-1}-1\|_{L^1} \leq \|\tau^n_{n-1}-\frac{n-1}{n}   \|_{L^2}+\frac{1}{n}\leq \frac{C+1}{n} \rightarrow 0$. 
Claim 2 is thus also proved.
\end{proof}

\begin{lemma}\label{prelim_lemma}
Let $(\Psi_n)$ be a sequence of convex functions that converges uniformly to a $C^1$ function $\Psi$. Let $(x_n)$ and $(z_n)$ be two real sequences such that:\\
(1) $x_n$ converges to $x$.\\
(2) for all $n$: $z_n \in \partial\Psi_n(x_n)$.\\
Then $z_n$ converges to $\Psi'(x)$.
\end{lemma}

\begin{proof}
Since $z_n \in \partial\Psi_n(x_n)$, we get with $u \in \{-1,+1\}$ that:
\[ \Psi(x_n+u)+\| \Psi-\Psi_n \|_{\infty} \geq \Psi_n(x_n+u) \geq \Psi_n(x_n)+u z_n \geq \Psi(x_n)-\| \Psi_n-\Psi\|_{\infty} + u z_n
\]
Therefore $ u z_n \leq \Psi(x_n+u) - \Psi(x_n) + 2 \| \Psi_n-\Psi\|_{\infty}$ and thus:

\[| z_n | \leq \displaystyle \max_{u \in \{-1,+1\}} \Psi(x_n+u) - \Psi(x_n) + 2 \| \Psi_n-\Psi\|_{\infty}\]
Since the right hand side is bounded, any subsequence of $(z_n)$ has an accumulation point. All these accumulation points must be in $\partial\Psi(x)$. Indeed, if a subsequence $(z_{n(k)})$ converges to $z$, we have for all $y$: $ \Psi_{n(k)}(y) \geq \Psi_{n(k)}(x_{n(k)})+z_{n(k)}(y-x_{n(k)})$. Letting $k$ go to infinity, we get then for all $y$: $\Psi(y) \geq \Psi(x)+z(y-x)$ and therefore $z \in \partial \Psi(x)= \{ \Psi'(x) \}$ since $\Psi$ is $C^1$.
All subsequence of $(z_n)$ has $\Psi'(x)$ as accumulation point, this is equivalent to the claim that $z_n$ converges to $\Psi'(x)$.
\end{proof}

\subsection{Convergence of the historical law}\label{subsection:convergence_Pn}

Let $(\Pi_n,X_n)$ be a sequence of reduced equilibria in $G_n(\mu)$. We already know that $X_n=\Psi_{\nu_n,\overline{\lambda}_n}(S_n^n)$ and that the marginal $\overline{\Pi}_n$
of $\Pi_n$ on $(L,S_n^n)$ coincides with $\Pi_{\nu_n}$ for a measure $\nu_n$ satisfying $T_{\overline{\lambda}_n}(\nu_n)=\overline{\lambda}_n$. We further know that $\nu_n$ converges to the unique solution $\nu$ of $T_{\overline{\lambda}_{\infty}}(\nu)=\overline{\lambda}_{\infty}$. Therefore, $\overline{\Pi}_n$ converges to $\Pi_{\nu}$ in $W_2$ distance. Our aim in this section is somehow to analyze the asymptotics of the law $\Pi_n$ of $(u_1,\dots,u_n,L)$, or more specifically, the law of the price process $p_q^n(u_1,\dots,u_{q-1})$ when $(u_1,\dots,u_n,L) \sim \Pi_n$. This law is called the historical law. Note that we can't speak of the convergence of $\Pi_n$ itself because the space on which probability $\Pi_n$ is defined depends on $n$. We will use the embedding technics introduced in the previous subsection.\\
 
Let $y_n(\omega_n)$ denote the density of $\frac{\partial\Pi_{n|\omega_n}}{\partial\lambda_n}$, so $y_n$ is a function of $\omega_n=(u_1,\dots,u_n)$. In the previous subsection, we created sequences $\tilde{S^n_q}=B_{\tau^n_q}$ and $\tilde{u}$ of random variables on $(\tilde{\Omega},\mathcal{A},\tilde{P})$ a probability space on which $B$ is a Brownian motion in such a way that $\tilde{S^n_q}$ and $\tilde{u}$ have the same distribution as $S_n$ and $u$ under $\lambda_n$.

Setting $\tilde{y}_n:=y_n(\tilde{u}_1,\dots,\tilde{u}_n)$, we infer that $\tilde{y}_n$ is a probability density on $(\tilde{\Omega},\mathcal{A},\tilde{P})$, and under the probability $\tilde{P}_n:=\tilde{y}_n . \tilde{P}$, the process $(\tilde{u}_1,\dots,\tilde{u}_n)$ is  $\Pi_{n | \omega_n}$-distributed.

We first prove the following lemma:

\begin{lemma}\label{lemma:y_n_to_y}
$\tilde{y}_n$ converges in $L^1(\tilde{P})$ to $\tilde{y}:=\frac{\beta}{\tilde{Y}}$ where $\tilde{Y}:=
H'(\Psi'_{\nu}(B_1)B_1-\Psi_{\nu}(B_1))$ and 
$\beta=\frac{1}{E_{\tilde{P}}[\frac{1}{\tilde{Y}}]}$
\end{lemma}

\begin{proof}
 
Our first task will be to define a variable $\tilde{L}_n$ on the space $(\tilde{\Omega},\mathcal{A},\tilde{P})$ such that the process $(\tilde{u}_1,\dots,\tilde{u}_n,\tilde{L}_n)$ is $\Pi_n$-distributed under $\tilde{P}_n$.

This can be done as follows: $\tilde{\omega}_n:=(\tilde{u}_1,\dots,\tilde{u}_n)$ is  $\mathcal{F}_{\tau_n^n}$ measurable. Let $V_n:=B_{\tau^n_{n}+1}-B_{\tau^n_{n}}$. Under $\tilde{P}$, $V_n \sim \mathcal{N}(0,1)$  and is independent of $\mathcal{F}_{\tau_n^n}$. Since $\tilde{y}_n=y_n(\tilde{\omega}_n)$, $V_n$ will have the same law $\mathcal{N}(0,1)$ and will still be independent of $\mathcal{F}_{\tau_n^n}$ under $\tilde{P}_n$. Let $F_{\omega_n}$ denote the cumulative distribution function of the conditional law of $L$ conditionally to $\omega_n$ under $\Pi_n$. We then set $\tilde{L}_n:=F_{\tilde{\omega}_n}^{-1}(F_{\mathcal{N}(0,1)}(V_n))$. $\tilde{L}_n$ has the same conditional law given $\tilde{\omega}_n$ as $L$ given $\omega_n$ under $\Pi_n$. Therefore $(\tilde{\omega}_n,\tilde{L}_n)$ under $\tilde{P}_n$ has the same law as $(\omega_n,L)$ under $\Pi_n$.\\

We now prove that, under $\tilde{P}$, $\tilde{L}_n$ converges to $\Psi'_{\nu}(B_1)$ almost surely. Indeed, since  $L_n$ belongs $\Pi_n$-almost surely to $\partial\Psi_n(S_n^n)$, we infer that $\tilde{L}_n$ belongs $\tilde{P}_n$-almost surely to $\partial\Psi_n(\tilde{S}_n^n)$. Since $\tilde{P}_n$ is equivalent to $\tilde{P}$, we conclude that $\tilde{L}_n \in \partial\Psi_n(\tilde{S}_n^n)$ $\tilde{P}$-almost surely.\\

Since $\Psi_n$ converges uniformly to $\Psi_{\nu} \in C^2$ and since $\tilde{S}_n^n$ converges  almost surely to $B_1$, we  apply Lemma \ref{prelim_lemma} to conclude that $\tilde{L}_n$ converges $\tilde{P}$-almost surely to $\Psi'_{\nu}(B_1)$. Therefore:

\begin{equation}\label{eq:cvg_couple}
(\tilde{L}_n,\tilde{S}_n^n) \buildrel{\tilde{P}-a.s.}\over\rightarrow (\Psi'_{\nu}(B_1),B_1)
\end{equation}

We define $Y_n:=E_{\Pi_n}[H'(LS_n^n-\Psi_n(S_n^n))|\omega_n]$. $Y_n$ is then a function $Y_n(\omega_n)$. It follows from Corollary \ref{Cor:C1234} that $\frac{\partial\lambda_n}{\partial\Pi_{n|\omega_n}}=\frac{Y_n}{E_{\Pi_n}[Y_n]}$.

We set $\tilde{Y}_n:=Y_n(\tilde{\omega}_n)$. We clearly have $\tilde{P}_n$-almost surely, and thus also $\tilde{P}$-almost surely, that $\tilde{Y}_n=E_{\tilde{P}_n}[H'(\tilde{L}_n\tilde{S}_n^n-\Psi_n(\tilde{S}_n^n)) | \tilde{\omega}_n]$. 

Since $\tilde{y}_n=\frac{\partial \tilde{P}_n}{\partial \tilde{P}}$ is just a function of $\tilde{\omega}_n$ , it follows that:
$$E_{\tilde{P}_n}[H'(\tilde{L}_n\tilde{S}_n^n-\Psi_n(\tilde{S}_n^n)) | \tilde{\omega}_n]= \frac{E_{\tilde{P}} [\tilde{y}_n H'(\tilde{L}_n\tilde{S}_n^n-\Psi_n(\tilde{S}_n^n)) |\tilde{\omega}_n]}{E_{\tilde{P}} [\tilde{y}_n |\tilde{\omega}_n]}=E_{\tilde{P}}[H'(\tilde{L}_n\tilde{S}_n^n-\Psi_n(\tilde{S}_n^n)) | \tilde{\omega}_n]$$

Note next that $\tilde{S}_n^n$ is $\tilde{\omega}_n$-measurable and $\tilde{L}_n$ is a function of the pair $(\tilde{\omega}_n,V_n)$. Since $V_n$ is independent of $\mathcal{F}_{\tau_n^n}$ we get then: $$\tilde{Y}_n= E_{\tilde{P}}[H'(\tilde{L}_n\tilde{S}_n^n-\Psi_n(\tilde{S}_n^n)) | \tilde{\omega}_n]=E_{\tilde{P}}[H'(\tilde{L}_n\tilde{S}_n^n-\Psi_n(\tilde{S}_n^n)) | \mathcal{F}_{\tau_n^n}]$$ 

We claim that $\tilde{Y}_n$ converge in $L^1$ to $\tilde{Y}:=H'(\Psi'_{\nu}(B_1)B_1-\Psi_{\nu}(B_1))$. Indeed:
\begin{align*}
\| \tilde{Y_n} - \tilde{Y} \|_{L^1}=& \| E_{\tilde{P}}[H'(\tilde{L}_n\tilde{S}_n^n-\Psi_n(\tilde{S}_n^n))| \mathcal{F}_{\tau^n_n}]-\tilde{Y} \|_{L^1} \\ \leq& \| E_{\tilde{P}}[H'(\tilde{L}_n\tilde{S}_n^n-\Psi_n(\tilde{S}_n^n))-\tilde{Y}| \mathcal{F}_{\tau^n_n}] \|_{L^1}+ \| E_{\tilde{P}}[\tilde{Y}| \mathcal{F}_{\tau^n_n}]-\tilde{Y} \|_{L^1}
\end{align*}

We next claim that  $H'(\tilde{L}_n\tilde{S}_n^n-\Psi_n(\tilde{S}_n^n))$ converges $\tilde{P}$-almost surely to $\tilde{Y}$. Indeed, according to equation (\ref{eq:cvg_couple}), $(\tilde{L}_n, \tilde{S}_n^n)$ converges almost surely to $(\Psi'_{\nu}(B_1),B_1)$. Furthermore $\Psi_n$ converges uniformly to $\Psi_{\nu}$. Therefore $\Psi_n(\tilde{S}_n^n)$ converges $\tilde{P}_n$- almost surely to $\Psi_{\nu}(B_1)$. Since $H'$ is continuous, we conclude that $H'(\tilde{L}_n\tilde{S}_n^n-\Psi_n(\tilde{S}_n^n))$ converges $\tilde{P}_n$-almost surely to $\tilde{Y}$.
Since $H'$ is further bounded, this later convergence holds also in $L^1$ and thus:
\[ \|E_{\tilde{P}}[H'(\tilde{L}_n\tilde{S}_n^n-\Psi_n(\tilde{S}_n^n))-\tilde{Y} | \mathcal{F}_{\tau_n^n}] \|_{L^1} \leq \|H'(\tilde{L}_n\tilde{S}_n^n-\Psi_n(\tilde{S}_n^n))-\tilde{Y} \|_{L^1} \rightarrow 0\]

We next claim that  $E_{\tilde{P}}[\tilde{Y}| \mathcal{F}_{\tau^n_n}]$ converges to $\tilde{Y}$ in $L^1$. On the contrary one would have a subsequence $n(k)$ such that $E_{\tilde{P}}[\tilde{Y}| \mathcal{F}_{\tau^{n(k)}_{n(k)}}]$ does not admit $\tilde{Y}$ as accumulation point in $L^1$.

Since $\mathcal{F}_t$ is the natural filtration of a Brownian motion, it results from theorem V.3.5 in \citet{revuz_yor} that the martingale $r_t:= E[\tilde{Y}| \mathcal{F}_t]$ is continuous and uniformly integrable. Therefore, due to the optional stopping theorem, we have $E_{\tilde{P}}[\tilde{Y}| \mathcal{F}_{\tau^{n(k)}_{n(k)}}]=  r_{\tau_{n(k)}^{n(k)}}$.

Since $\tau_n^n$ converges in $L^2$ to $1$, there is no loss of generality to assume, possibly after selection of a smaller subsequence, that $n(k)$ further satisfies that $\tau_{n(k)}^{n(k)}$ converges $\tilde{P}$-almost surely to $1$. But then $r_{\tau_{n(k)}^{n(k)}}$ converges almost surely to $r_1=E_{\tilde{P}}[\tilde{Y}| \mathcal{F}_1]=\tilde{Y}$. But due to the  uniform integrability of the martingale $r_t$, this convergence also holds in $L^1$, in contradiction with the definition of the subsequence $n(k)$. Therefore, as announced, $E_{\tilde{P}}[\tilde{Y}| \mathcal{F}_{\tau^n_n}]$ converges to $\tilde{Y}$ in $L^1$.\\

According to Corollary \ref{Cor:C1234}, $\frac{\partial \lambda_n}{\partial \Pi_{n|\omega}}= \alpha_n . Y_n$, and thus $\frac{\partial \Pi_{n|\omega_n}}{\partial\lambda_n}= \frac{\beta_n}{Y_n}$ for a constant $\beta_n$. Therefore for all $\omega_n$, $y_n(\omega_n)=\frac{\beta_n}{Y_n(\omega_n)}$ and $\tilde{y}_n=\frac{\beta_n}{\tilde{Y}_n}$. Since $\tilde{Y}_n$ is a probability density under $\tilde{P}$ we get $\beta_n=\frac{1}{E_{\tilde{P}}[\frac{1}{\tilde{Y}_n}]}$. 

Since $0<\epsilon<\tilde{Y}<K$ (assumptions \textbf{A2} on $H$), $\frac{1}{\tilde{Y}_n}$ converges in $L^1$ to $\frac{1}{\tilde{Y}}$ and it results as announced that  $\tilde{y}_n$ converges in $L^1$ to $\tilde{y}=\frac{\beta}{\tilde{Y}}$ where $\beta=\frac{1}{E_{\tilde{P}}[\frac{1}{\tilde{Y}}]}$.
\end{proof}

Theorem \ref{thm:unique_equiv_measure} claims that the martingale equivalent distribution $Q_n$ converges to a limit distribution $Q$. The next theorem is the counterpart of this result for the historical distribution. It claims that $P_n$ converges to a limit distribution $P$ which is the law of the process $Z$ when $\tilde{\Omega}$ is endowed with the probability measure $\tilde{y}\tilde{P}$. Therefore the limit distributions $P$ and $Q$ are equivalent.

This result is the main result of this paper. It claims that the asymptotics of the historical price process is a CMMV under an appropriate martingale equivalent measure $Q$.

\begin{theorem}\label{thm:histo}
The price process $Z^n_t$ under the probability $\Pi_n$ converges in finite dimensional distribution to the process $Z$ when $\tilde{\Omega}$ is endowed with the probability $\tilde{y}.\tilde{P}$ where $\tilde{y}=\frac{1}{E_{\tilde{P}}[\frac{1}{\tilde{Y}}]\tilde{Y}}>0$.
\end{theorem}

\begin{proof}
Let $J$ a finite family of times. Let $\phi$ be a continuous and bounded function: $\mathbb{R}^{|J|} \rightarrow \mathbb{R}$. 
It is convenient to introduce the notations  $\tilde{Z}^n_J:=(\tilde{Z_t^n})_{t \in J}$ and $Z^n_J:=(Z_t^n)_{t \in J}$. Then observe that $E_{\Pi_n}[\phi(Z_J^n)]= E_{\tilde{P}_n}[\phi(\tilde{Z_J^n})]=E_{\tilde{P}}[\tilde{y}_n\phi(\tilde{Z_J^n})]$. We next claim that $E_{\tilde{P}}[\tilde{y}_n\phi(\tilde{Z_J^n})]$  converges to $E_{\tilde{P}}[\tilde{y}\phi({Z_J})]$. Indeed, on the contrary there would exist a subsequence $n(k)$ that $E_{\tilde{P}}[\tilde{y}_{n(k)}\phi(\tilde{Z}_J^{n(k)})]$ does not admit $E_{\tilde{P}}[\tilde{y}\phi({Z_J})]$ as accumulation point. However, as it results from equation (\ref{eq:2conditions}) and Lemma \ref{lemma:y_n_to_y}, we have that $\tilde{Z}^{n(k)}_t \rightarrow Z_t$ in $L^2$ for all $t$ and that $\tilde{y}_{n(k)} \rightarrow \tilde{y}$ in $L^1$. Possibly after selection of a smaller subsequence, we may assume without loss of generality that the sequence $n(k)$ is further such that $\tilde{Z}^{n(k)}_J \rightarrow Z_J$ and that $\tilde{y}_{n(k)} \rightarrow \tilde{y}$ almost surely. Due to the continuity of $\phi$, we get that $\tilde{y}_{n(k)}\phi(\tilde{Z}_J^{n(k)})$ converges almost surely to $\tilde{y}\phi({Z_J})$. Since both $\phi$ and $\tilde{y}_n$ are bounded, we have with Lebesgue dominated convergence theorem that 
$E_{\tilde{P}}[\tilde{y}_{n(k)}\phi(\tilde{Z_J^{n(k)}})]$ converges to $E_{\tilde{P}}[\tilde{y}\phi({Z_J})]$, which is in contradiction with the definition of $n(k)$. Therefore, as announced, $E_{\Pi_n}[\phi(Z_J^n)] \rightarrow E_{\tilde{P}}[\tilde{y}\phi({Z_J})]$ for all $J$: the law of $Z_J^n$ converges weakly in $\Delta(\mathbb{R}^{|J|})$ to the law of $Z_J$ under $\tilde{y}.\tilde{P}$ and the theorem is proved.
\end{proof}

\section{Conclusion}

To conclude this paper we would like to make some remarks on the obtained results.\\

The first one is about the dual game. Our first attempt to analyze this game was using duality techniques. The dual game $G_n^{\star}(\phi)$ is in fact the reduced game where Player 1 is allowed to select privately the value of $L$ but his payoff is decreased by a penalty $\phi(L)$. The function $\phi$ is known by both players. Strategies and payoffs are the same for Player 2. A strategy $\Pi$ for Player 1 is a joint probability on $(\omega,L)$ but there is no constraint on the marginal $\Pi_{\mid L}$. It can be easily proved that if $(\Pi^{\star},\overline{p})$ is an equilibrium in $G_n^{\star}(\phi)$ and if $\mu = \Pi^{\star}_{\mid L}$ then $(\Pi^{\star},\overline{p})$ is an equilibrium in $G_n(\mu)$. It can then be proved that there exists a function $\phi_n$ and an equilibrium $(\Pi_n^{\star},\overline{p}_n)$ in $G_n^{\star}(\phi_n)$ such that $ \Pi^{\star}_{\mid L}=\mu$. Therefore $(\Pi_n^{\star},\overline{p}_n)$ is a sequence of equilibria in $G_n(\mu)$. One of the reason for introducing the dual game was that the asymptotics of the reduced equilibria in
$G_n^{\star}(\phi)$ was quiet easy to analyze (with $\phi$ independent of $n$). However, to analyze the asymptotics of the equilibria in $G_n(\mu)$ using the dual game, we would have to analyze a sequence of equilibria in $G_n^{\star}(\phi_n)$ for an appropriate sequence of $\phi_n$. This makes the analysis more involved and explains why we decided to limit our paper to the game $G_n(\mu)$.\\

The second remark is about the generality of our results. The results obtained in \citep{de2010} were somehow more general than those obtained in the present paper: in the risk-neutral case, if the mechanism belongs to the class of natural mechanisms, then the price process at equilibrium converge to a CMMV for all sequences of equilibria in $G_n(\mu)$. The current paper is only concerned with one particular natural mechanism for which the price process is explicit. For this mechanism we do not analyze the asymptotic of any sequence of equilibria, but only of sequences of reduced equilibria: we prove that the price processes at a reduced equilibrium converges to a CMMV under the risk-neutral probability. This naturally raises two questions: do we have the same asymptotic for any sequence of equilibria in our game? And will this dynamic appear for more general price based mechanism? We conjecture a positive answer to both questions but are presently unable to prove it.\\

Finally, we just want to mention an alternative approach to our results. It would indeed be possible to introduce continuous time games quite similar to the Brownian games introduced in \cite{de99} : a strategy $\Pi_n$ in the reduced game can be viewed as a pair $(y_n,\rho_n)$ where $\rho_n$ is a conditional law of $L$ given $\omega$ and $y_n$ is the density $\frac{\partial \Pi_{n | \omega}}{\partial \lambda_n}$. Player 1's payoff is given by $E_{\lambda_n}[y_n (L S_n - \sum_{q=1}^n p_q(S_{q}-S_{q-1})]$. Heuristically, under $\lambda_n$, $S_n$ converges to $B_1$ and the payoff function of Player 1 should converge to $E[ y(L B_1  - \int_0^1 p_t dB_t)]$. Similarly player 2 payoff would be $E[ y H( L B_1  - \int_0^1 p_t dB_t)]$.

\section{Annexes}

\subsection{Annexes for section \ref{section_conditions}}\label{annexe_6}

\begin{lemma}\label{lemma:p_q_is_increasing}
Let $(\Pi_n^{\star},\Psi_n^{\star})$ be an equilibrium and $\overline{p}$ the corresponding pure strategy of player 2.
Then we have:
\[\overline{p}^n_q(u_1,\dots,u_{q-2},1)>\overline{p}^n_q(u_1,\dots,u_{q-2},-1) \]

\end{lemma}

\begin{proof}
Since $\Psi_n$ is convex, its derivative exists except on a countable number of points. We may therefore write:
\begin{equation}\label{def:chi}
\chi(x):= \frac{\sqrt{n}}{2} \left(\Psi_n(x + \frac{1}{\sqrt{n}})-\Psi_n(x - \frac{1}{\sqrt{n}})\right) 
=\frac{\sqrt{n}}{2}\int_{\frac{-1}{\sqrt{n}}}^{\frac{1}{\sqrt{n}}} \Psi_n'(x+v)dv
\end{equation}

Denote $U=\{\frac{-n+1+2k}{\sqrt{n}} | k \in 0,\dots,n-1\}$ the set of possible values of $S_{n-1}^n$, where $S_q^n:=\frac{1}{\sqrt{n}}\sum_{k=1}^q u_q$.

 We first prove that $\chi$ is strictly increasing on $U$. Let $x,y$ be two successive points in $U$ (i.e. $y=x+\frac{2}{\sqrt{n}}$) and assume that $\chi(x)=\chi(y)$. We have then:
$$0=\chi(y)-\chi(x)=\frac{\sqrt{n}}{2}\int_{\frac{-1}{\sqrt{n}}}^{\frac{1}{\sqrt{n}}} \Psi_n'(y+v)-\Psi_n'(x+v)dv$$
Since $\Psi'$ is increasing, this is the integral of a positive function which is $0$. This is only possible if $ \Psi_n'(y+v)-\Psi_n'(x+v)=0$ for almost every $v$, which means in particular that $\Psi_n'(z)=\Psi_n'(z+\frac{2}{\sqrt{n}})$ for every $z \in ]x-\frac{1}{\sqrt{n}},x+\frac{1}{\sqrt{n}}[$. $\Psi_n'$ is therefore
constant on the interval $]x-\frac{1}{\sqrt{n}},y+\frac{1}{\sqrt{n}}[$. This implies that $\Psi$ is differentiable at the point $z=x+\frac{1}{\sqrt{n}}$ which is a possible value for $S_n=S_{n-1}^n+\frac{u_n}{\sqrt{n}}$. The event $\{ S_n = z \}$ has a strictly positive probability under $\lambda_n$ and we infer also that $\Pi_{n}(S_n = z)>0$ since $\lambda_n$ and $\Pi_{n \mid \omega}$ are equivalent probabilities as it follows from condition \textbf{(C4)} and the fact that $H'$ is strictly positive.

Since $\partial\Psi_n(z)=\{ \Psi_n'(z) \}$, we infer from  \textbf{(C3)} that $\Pi_n(L = \Psi'(z)) >0$. According to \textbf{(C2)}, $L$ is $\mu$-distributed under $\Pi_n$ and we would therefore have $\mu(\{ \Psi_n'(z) \})>0$, which contradict our hypothesis \textbf{A1} that $\mu$ has a density with respect to Lebesgue measure. This concludes the proof that $\chi$ is increasing on $U$.

According to formula (\ref{eq:price}), we get:

\begin{align*}
\overline{p}^n_q(u_1,\dots,u_{q-1}) =&  E_{\lambda_n} \left[ \frac{\sqrt{n}}{2} \left(\Psi_n(S_{n-1}^n + \frac{1}{\sqrt{n}})-\Psi_n(S_{n-1}^n - \frac{1}{\sqrt{n}})\right)  \mid u_1, \dots, u_{q-1} \right] \\
=&E_{\lambda_n}[\chi(S_{n-1}^n)|u_1,\dots,u_{q-1}]\\
=&E_{\lambda_n}[\chi(S_{n-1}^n)|S_{q-1}^n]\\
=&E_{\lambda_n}[\chi(S_{q-1}^n+V)|S_{q-1}^n]\\
\end{align*} 
where $V=\frac{u_q + \dots + u_{n-1}}{\sqrt{n}}$. Since $V$ is independent of $S_{q-1}$, we get therefore $\overline{p}^n_q(u_1,\dots,u_{q-1})=r(S_{q-1})$ where $r(x):=E_{\lambda_n}[\chi(x+V)]$ is a strictly increasing function on the set of possible values of $S^n_{q-1}$.
\end{proof}

\subsection{The continuity of $T_{\lambda}$}\label{continuity_annexe}

The aim of this subsection is to prove Proposition \ref{cor:continuity}, which is usefull to prove the existence of an equilibrium in section  \ref{section_existence}.\\

We will prove the continuity of the operator step by step. The following lemma are useful in the proof of the continuity.

\begin{lemma}\label{lemma:Continuity_Phi_and_Gamma}
The mappings $\nu \rightarrow \Phi_{\nu}$ and $\nu \rightarrow \Gamma_{\nu}$ are continuous from $(P_2(\mathbb{R}),W_2)$ to the set of convex functions on respectively $]0,1[$ and $\mathbb{R}$ with the topology of uniform convergence.
\end{lemma}

\begin{proof}
Let $\nu_1,\nu_2$ be two measures in $P_2(\mathbb{R})$ and $x \in ]0,1[$.

\begin{align*}
| \Phi_{\nu_1}(x)-\Phi_{\nu_2}(x) |  &\leq  \left| \int_{0}^x  F_{\nu_1}^{-1} (F_{\mu} (\ell)) -F_{\nu_2}^{-1} (F_{\mu} (\ell)) d\ell \right| \\
&\leq \int_{0}^x \left| F_{\nu_1}^{-1} (F_{\mu} (\ell)) -F_{\nu_2}^{-1} (F_{\mu} (\ell))\right| d\ell\\
&\leq \int_0^1 \left| F_{\nu_1}^{-1} (F_{\mu} (\ell)) -F_{\nu_2}^{-1} (F_{\mu} (\ell))\right| \frac{f_{\mu}(\ell)}{f_{\mu}(\ell)}d\ell\\
&\leq E_{\mu} [  \frac{ \left| F_{\nu_1}^{-1} (F_{\mu} (L)) -F_{\nu_2}^{-1} (F_{\mu} (L))\right|}{f_{\mu}(L)}]\\
&\leq \sqrt{ E_{\mu} [   ( F_{\nu_1}^{-1} (F_{\mu} (L)) -F_{\nu_2}^{-1} (F_{\mu} (L))) ^2 }] \sqrt{E_{\mu}[\frac{1}{ f_{\mu}(L)^2}] }
\end{align*}
The last inequality follows from Cauchy Scwharz theorem. The right hand side does not depend on $x \in ]0,1[$, therefore:

\begin{equation}\label{eq:unifpsi}
\| \Phi_{\nu_1}-\Phi_{\nu_2} \|_{\infty}  \leq \sqrt{ E_{\mu} [   ( F_{\nu_1}^{-1} (F_{\mu} (L)) -F_{\nu_2}^{-1} (F_{\mu} (L))) ^2 }] \sqrt{E_{\mu}[\frac{1}{ f_{\mu}(L)^2}] }
\end{equation}



It results for work of \citep{dall} and \citep{frechet1957distance} that $W_2(\nu_1,\nu_2)$ may be explicitly computed in the case of one dimensional probability distributions: minimizing $\int_\mathbb{R}  \mid x - y \mid ^ 2 d\pi(x,y)$ is equivalent to maximizing $E_{\pi}(xy)$, which amounts to maximizing $cov(XY)$ with $X \sim \nu_1$ and $Y \sim \nu_2$. This maximum is reached when $X$ and $Y$ can be written as increasing functions (here $F_{\nu_1}^{-1}$ and $F_{\nu_2}^{-1}$) of the same uniform random variable $U$. Indeed $U=F_{\mu}(L)$ is uniformly distributed on $[0,1]$ when $L \sim \mu$. Therefore we have that:
\begin{equation}\label{lemma:Brenier}
W_2(\nu_1,\nu_2)=\sqrt{ E_{\mu} [   ( F_{\nu_1}^{-1} (F_{\mu} (L)) -F_{\nu_2}^{-1} (F_{\mu} (L))) ^2 }]
\end{equation}

Since $f_{\mu}$ is bounded from below by $\epsilon>0$ ( assumption \textbf{A1}) we conclude:

\[
\| \Phi_{\nu_1}-\Phi_{\nu_2} \|_{\infty}  \leq W_2(\nu_1,\nu_2) \sqrt{E_{\mu} [\frac{1}{ f_{\mu}(L)^2} ]}\leq W_2(\nu_1,\nu_2) \sqrt{\frac{1}{ \epsilon^2}} \]

Then we proved that the mapping  $\nu \rightarrow \Phi_{\nu}$ is $\sqrt{\frac{1}{ \epsilon^2}}$-Lipschitz continuous for the uniform norm.

We now prove that $\nu \rightarrow \Gamma_{\nu}$ is also Lipschitz continuous.

Observe that $\Gamma_{\nu}(s) = \Phi_{\nu}^{\sharp}(s)-\Phi_{\nu}^{\sharp}(0)$. Indeed from the definition of $\Gamma_{\nu}$ and $\Phi_{\nu}$ we get that the $\partial\Phi_{\nu}(\ell)=[\phi_{\nu}(\ell^-),\phi_{\nu}(\ell)]$ and thus by Fenchel lemma :
\[ \partial\Phi_{\nu}^{\sharp}(s)=[\phi_{\nu}^{-1}(s^-),\phi_{\nu}^{-1}(s)]=[\gamma_{\nu}(s^-),\gamma_{\nu}(s)]=\partial \Gamma_{\nu}(s) \]

The two functions $\Phi_{\nu}^{\sharp}$ and $\Gamma_{\nu}$ just differ by a constant, and since $\Gamma_{\nu}(0)=0$ we find $\Gamma_{\nu}(x)=\Phi_{\nu}^{\sharp}(x)-\Phi_{\nu}^{\sharp}(0)$. As well known Fenchel transform in an isometry for the uniform norm\footnote{Let indeed $f$ and $g$ be two lower semi continuous convex functions $\mathbb{R}^n \rightarrow \mathbb{R}$, then $\mid\mid f^{\sharp}-g^{\sharp} \mid\mid_{\infty} = \mid\mid f-g \mid\mid_{\infty}$. Indeed, for all $x \in \mathbb{R}$: $f^{\sharp}(x)  =  \sup_t   x.t - f(t) \leq  \sup_t  x.t - g(t)  + \| f-g \|_{\infty}= g^{\sharp}(x)+ \| f-g \|_{\infty}$. Interchanging $f$ and $g$ we get therefore for all $x$: $  |f^{\sharp}(x) - g^{\sharp}(x) |   \leq \| f-g \|_{\infty}$. Since the right hand side doesn't depend on $x$, we get: $\| f^{\sharp}-g^{\sharp} \|_{\infty} \leq \| f-g \|_{\infty}$. The reverse inequality follows from Fenchel lemma: $f^{\sharp \sharp}=f$ and $g^{\sharp \sharp}=g$. Therefore: $\| f-g \|_{\infty} = \| f^{\sharp \sharp}-g^{\sharp \sharp} \|_{\infty} \leq  \| f^{\sharp}-g^{\sharp} \|_{\infty}$ as announced.}. We conclude that the mapping $\nu \rightarrow \Gamma_{\nu}$ is also Lipschitz continuous for the uniform norm.
\end{proof}

\begin{lemma}\label{continuity_of_Pi_n}
If $W_2(\nu_k,\nu) \rightarrow 0$ then $W_2(\Pi_{\nu_k},\Pi_{\nu}) \rightarrow 0$.
\end{lemma}

\begin{proof}

Let $L$ be a random variable with law $\mu$.
Let $X_1=(L,\phi_{\nu_k}(L))$ and  $X_2=(L,\phi_{\nu}(L))$. Then $X_1 \sim \Pi_{\nu_k}$ and $X_2 \sim \Pi_{\nu}$.
\[ W_2(\Pi_{\nu_k},\Pi_{\nu})^2 \leq \| X_1-X_2 \|^2_{L^2} =  \| L-L \|^2_{L^2} + \| \phi_{\nu_k}(L)-\phi_{\nu}(L) \|^2_{L^2} = W_2(\nu_k,\nu)^2 \]

where the last equality follows from equation (\ref{lemma:Brenier}).
\end{proof}

\begin{lemma}\label{le:12continuity}
If $W_2(\lambda_k,\lambda) \rightarrow 0$ and $W_2(\nu_k,\nu) \rightarrow 0$ then:

1/ $\| \Psi_{\nu_k,\lambda_k}-\Psi_{\nu,\lambda} \|_{\infty} \rightarrow 0$

2/ For all continuous function $\Theta$ such that $\frac{\Theta(x)}{1+x^2}$ is bounded, we have:

\[ E_{\Pi_{\nu_n}}[\Theta(S_n)H'(S_n L - \Psi_{\nu_k,\lambda_k}(S_n))] \rightarrow E_{\Pi_{\nu}}[\Theta(S_n)H'(S_n L - \Psi_{\nu,\lambda}(S_n))]
\]
\end{lemma}

\begin{proof}
\begin{align*}
1/ ~~ \| \Psi_{\nu_k,\lambda_k}-\Psi_{\nu,\lambda} \|_{\infty} =& \| (\Gamma_{\nu_k} - E_{\lambda_k}[\Gamma_{\nu_k}]) - (\Gamma_{\nu} - E_{\lambda}[\Gamma_{\nu}] ) \|_{\infty}\\
\leq& \| \Gamma_{\nu_k}-\Gamma_{\nu} \|_{\infty} +  \| E_{\lambda_k}[\Gamma_{\nu_k}]-E_{\lambda_k}[\Gamma_{\nu}] \|_{\infty}+|E_{\lambda_k}[\Gamma_{\nu}]-E_{\lambda}[\Gamma_{\nu}] |\\
\leq& 2 \| \Gamma_{\nu_k}-\Gamma_{\nu} \|_{\infty} +|E_{\lambda_k}[\Gamma_{\nu}]-E_{\lambda}[\Gamma_{\nu}] |
\end{align*}

The first term of the right hand side goes to zero according to Lemma \ref{lemma:Continuity_Phi_and_Gamma}. Next observe that $\partial \Gamma_{\nu}=[\gamma(s^-),\gamma(s)] \subset [0,1]$. Indeed,
$\gamma_{\nu}(s)=F_{\mu}^{-1} (F_{\nu} (s))$ belongs to $[0,1]$ since the support of $\mu$ is $[0,1]$. Furthermore, by definition $\Gamma_{\nu}$ is a continuous function that satisfies $\Gamma_{\nu}(0)=0$. Therefore, $|\Gamma_{\nu}(x)|\leq |x| \leq C(1+x^2)$ for a constant $C$. Since $W_2(\nu_k,\nu)$ goes to zero, we conclude with Proposition \ref{prop:W_2_for_weak} that $\nu_k$ converges weakly in $P_2(\mathbb{R})$ and therefore $E_{\lambda_k}[\Gamma_{\nu}] \rightarrow E_{\lambda}[\Gamma_{\nu}]$

 $\Gamma$ is further continuous as claimed in Proposition \ref{prop:W_2_for_weak}, the last term goes also to zero.

 $$ 2/ ~~
| E_{\Pi_{\nu_k}}[\Theta(S_n)H'(S_n L - \Psi_{\nu_k,\lambda_k}(S_n))] - E_{\Pi_{\nu}}[\Theta(S_n)H'(S_n L - \Psi_{\nu,\lambda}(S_n))] | \leq I_k + J_k
$$ where $$I_k:=| E_{\Pi_{\nu_k}}[\Theta(S_n)H'(S_n L - \Psi_{\nu_k,\lambda_k}(S_n))]- E_{\Pi_{\nu_k}}[\Theta(S)H'(S_n L - \Psi_{\nu,\lambda}(S_n))]|$$ and

$$J_k:= | E_{\Pi_{\nu_k}}[\Theta(S_n)H'(S_n L - \Psi_{\nu,\lambda}(S_n))]- E_{\Pi_{\nu}}[\Theta(S_n)H'(S_n L - \Psi_{\nu,\lambda}(S_n))] |.$$
According to \textbf{A1}, $H'$ is Lipschitz continuous. Let $\hat{K}$ denote the Lipschitz constant, then:
\[I_k \leq E_{\Pi_{\nu_k}}[| \Theta(S_n) | \hat{K} \|\Psi_{\nu_k,\lambda_k}-\Psi_{\nu,\lambda} \|_{\infty}]=\hat{K} \|\Psi_{\nu_k,\lambda_k}-\Psi_{\nu,\lambda} \|_{\infty}E_{\nu_k}[| \Theta(S_n) | ]\]
Since $| \Theta(S_n) |\leq C(1+S_n^2)$ and $W_2(\nu_k,\nu)\rightarrow 0$ we get with Proposition \ref{prop:W_2_for_weak} that $E_{\nu_k}[| \Theta(S_n) | ] \rightarrow E_{\nu}[| \Theta(S_n) | ]< \infty$. On the other hand  $\|\Psi_{\nu_k,\lambda_k}-\Psi_{\nu,\lambda} \|_{\infty} \rightarrow 0$ according to the first claim of this lemma. $I_k$ converges therefore to 0.\\

The map $(L,S_n) \rightarrow \Theta(S_n)H'(S_n L - \Psi_{\nu,\lambda}(S_n))$ is continuous and is also bounded by $C(1+\|(L,S_n)\|^2)$ since $H'$ is itself bounded. Since $\Pi_{\nu_k}$ converges to $\Pi_{\nu}$ in $W_2$, it follows from Proposition \ref{prop:W_2_for_weak} that $J_k$ goes to zero.
\end{proof}

We are now ready to prove the Proposition \ref{cor:continuity}. 

\begin{proof}[Proof of Proposition \ref{cor:continuity}.]~~\\
We have to prove that if $\Theta$ continuous and satisfies $\mid \Theta(s) \mid \leq C(1+s^2)$, then:

\[ E_{T_{\lambda_k}(\nu_k)}[\Theta] \rightarrow E_{T_{\lambda}(\nu)}[\Theta] \]

According to the definition of $T_{\lambda_k}(\nu_k)$ this amounts to show that:
 
\begin{equation}\label{eq_avec_c}
E_{\nu_k}[\Theta(S_n).\alpha_{\nu_k,\lambda_k}. Y_{\nu_k,\lambda_k}(S_n)] \rightarrow E_{\nu}[\Theta(S_n).\alpha_{\nu,\lambda}. Y_{\nu,\lambda}(S_n)]
\end{equation}

We first prove that:

\begin{equation}\label{eq_sans_c}
E_{\nu_k}[\Theta(S_n). Y_{\nu_k,\lambda_k}(S_n)]\rightarrow E_{\nu}[\Theta(S_n). Y_{\nu,\lambda}(S_n)]
\end{equation}

Due to the definition of $Y_{\nu_k,\lambda_k}$ we get:

\begin{align*}
E_{\nu_k}[\Theta(S_n). Y_{\nu_k,\lambda_k}(S_n)]=&
E_{\nu_k}[\Theta(S_n).E_{\Pi_{\nu_k}}[H'(LS_n-\Psi_{\nu_k,\lambda_k}(S_n))|S_n]]\\
=&E_{\Pi_{\nu_k}}[\Theta(S_n).E_{\Pi_{\nu_k}}[H'(LS_n-\Psi_{\nu_k,\lambda_k}(S_n))|S_n]]\\
=&E_{\Pi_{\nu_k}}[\Theta(S_n).H'(LS_n-\Psi_{\nu_k,\lambda_k}(S_n))]
\end{align*}
and we have a similar formula for $E_{\nu}[\Theta(S_n). Y_{\nu,\lambda}(S_n)]$. According to claim 2 in Lemma \ref{le:12continuity}:

$$E_{\Pi_{\nu_k}}[\Theta(S_n).H'(LS_n-\Psi_{\nu_k,\lambda_k}(S))] \rightarrow E_{\Pi_{\nu}}[\Theta(S_n).H'(LS_n-\Psi_{\nu,\lambda}(S_n))]$$
and formula (\ref{eq_sans_c}) follows.

According to the Definition \ref{def:Y_n} we get $\alpha_{\nu_k,\lambda_k}=\frac{1}{E_{\nu_k}[Y_{\nu_k,\lambda_k}]}$. With formula (\ref{eq_sans_c}) for the particular $\Theta \equiv 1$, we get that $E_{\nu_k}[Y_{\nu_k,\lambda_k}] \rightarrow E_{\nu}[Y_{\nu,\lambda}]$. Since $Y_{\nu,\lambda}$ is bounded from below by $\epsilon>0$ (assumption \textbf{A2} on $H$), we conclude then that $\alpha_{\nu_k,\lambda_k} \rightarrow \alpha_{\nu,\lambda}$. Finally combining this result with formula (\ref{eq_sans_c}), we get the convergence announced in formula (\ref{eq_avec_c}) and the corollary is proved.
\end{proof}

\subsection{Proof of Proposition \ref{prop:equiv}}\label{Annexe:equivalence}


\begin{proof}
Let $\nu$ satisfy the equation $T_{\overline{\lambda}_{\infty}}(\nu)=\overline{\lambda}_{\infty}$. This implies that $\overline{\lambda}_{\infty}$ has a strictly positive density with respect to $\nu$, and therefore $\nu$ has a density with respect to $\overline{\lambda}_{\infty}$. In turn this implies that it has also a density $f_{\nu}$ with respect to the Lebesgue measure. This implies in particular that $f_{\nu}$ is continuous.

We first deal with the smoothness of $\Psi_{\nu,\overline{\lambda}_{\infty}}$. Remember that  $\Psi_{\nu,\overline{\lambda}_{\infty}}$ differs from $\Gamma_{\nu}$ just by a constant. $\Gamma_{\nu}$ was defined as an integral of $\gamma_{\nu}(s)=F_{\mu}^{-1}(F_{\nu}(s))$. Since $F_{\mu}$ is a strictly increasing and continuous map from $[0,1]$ to $[0,1]$, its inverse is itself continuous. Since $F_{\nu}$ is also continuous, it follows that $\Gamma_{\nu}$ and $\Psi_{\nu,\overline{\lambda}_{\infty}}$ are $C^1$ and $\Psi_{\nu,\overline{\lambda}_{\infty}}'=\gamma_{\nu}$.

We next prove that $\gamma_{\nu}$ is absolutely continuous\footnote{A function $f:\mathbb{R}\rightarrow\mathbb{R}$ is absolutely continuous (see definition 7.17 in \citet{rudin}) if for all $\epsilon>0$, and for all sequences of disjoint real intervals $[a_n,b_n]$, there exists $\delta$ such that:
\[ \sum_{n \geq 0} |b_n-a_n|< \delta \Rightarrow \sum_{n \geq 0} |f(a_n)-f(b_n)| < \epsilon \] 

}. This will imply on one hand  (see theorem 7.18 in \citet{rudin}) the existence of a function $g$ integrable with respect to the Lebesgue measure such that $\gamma_{\nu}(s)=\int_{-\infty}^s g(t)dt$ and on the other hand (by the Lebesgue differentiation theorem) that $\gamma_{\nu}$ is almost surely differentiable and for almost every $s$: $\gamma_{\nu}'(s)=g(s)$. The first claim of the proposition will then be proved by establishing that $g$, which is only defined up to negligeable set, admits a continuous version.

\underline{$\nu$ is absolutely continuous:}\\
Since $\nu$ is absolutely continuous with respect to the Lebesgue measure, its cumulative distribution function $F_{\nu}$ is an absolutely continuous function\footnote{If $\nu$ is absolutely continuous with respect to the Lebesgue measure, we have $F_{\nu}(x)-F_{\nu}(a)=\int_a^x f_{\nu}(t)dt$. Theorem 7.18 in \citet{rudin} implies that $F_{\nu}$ is absolutely continuous.}.

We next observe that $F_{\mu}^{-1}$ is Lipschitz continuous. According to \textbf{A1}, $f_{\nu}$ is $C^1$ on [$0,1]$ and strictly positive. Let then $\kappa>0$ be such that $\kappa < f_{\mu}$. For $\tilde{s}<s$, we set $b=F_{\mu}^{-1}(s)$ and $\tilde{b}=F_{\mu}^{-1}(\tilde{s})$ then:

\[ | s - \tilde{s} | = s - \tilde{s} = F_{\mu}(b)-F_{\mu}(\tilde{b})= \int_{\tilde{b}}^b f_{\mu}(x)dx \geq \kappa (b-\tilde{b})
\] 

Therefore, we have:

\[   |F_{\mu}^{-1}(s)-F_{\mu}^{-1}(\tilde{s})|  \leq \frac{1}{\kappa} |s-\tilde{s}|.
\]

The function $\gamma_{\nu}=F_{\mu}^{-1}(F_{\nu}(s))$ introduced in Definition \ref{def:Phi} is therefore absolutely continuous. Indeed, since $F_{\nu}$ is absolutely continuous, for $\epsilon>0$ and $[a_n,b_n]$ disjoint real intervals, there exists $\delta$ such that:
\[ \sum_{n \geq 0} |b_n-a_n|< \delta \Rightarrow \sum_{n \geq 0} |(F_{\nu}(a_n)-F_{\nu}(b_n)| < \kappa \epsilon \]

Suppose that  $\sum_{n \geq 0} |b_n-a_n|< \delta$. Then:

\[ \sum_{n \geq 0} |F_{\mu}^{-1}(F_{\nu}(a_n))-F_{\mu}^{-1}(F_{\nu}(b_n))| \leq \sum_{n \geq 0} \frac{1}{\kappa} |(F_{\nu}(a_n)-F_{\nu}(b_n)| \leq \epsilon.
\]

Since $f_{\mu}$ is $C^1$ (see conditions \textbf{A1}) and positive, $F_{\mu}^{-1}$ is itself $C^1$ and $F_{\mu}^{-1\prime}(u)=\frac{1}{f_{\mu}(F^{-1}(u))}$. Since $F_{\nu}$ is absolutely continuous, it is almost surely differentiable and $F_{\nu}'(s)=f_{\nu}(s)$, according to Lebesgue differentiation theorem. Therefore, with the composition rule, we conclude:

\begin{equation}\label{eq:eq_diff_density}
g(s)=\gamma_{\nu}'(s)=\frac{f_{\nu}(s)}{f_{\mu}(F_{\mu}^{-1}(F_{\nu}(s)))}=\frac{f_{\nu}(s)}{f_{\mu}(\gamma_{\nu}(s))}
\end{equation}

Since $\Psi_{\nu,\overline{\lambda}_{\infty}}$ is $C^1$, and $\overline{\Pi}_{\nu}$ satisfies \textbf{(C4)}: $\overline{\Pi}_{\nu}(L \in \partial\Psi_{\nu,\overline{\lambda}_{\infty}}(S))=1$, we infer that $L$ is almost surely equals to $\Psi_{\nu,\overline{\lambda}_{\infty}}^{\prime}(S)$ under $\overline{\Pi}_{\nu}$ and thus:

 $$E_{\overline{\Pi}_{\nu}} [H'(LS-\Psi_{\nu,\overline{\lambda}_{\infty}}(S))|S]=H'(\Psi_{\nu,\overline{\lambda}_{\infty}}^{\prime}(S)S-\Psi_{\nu,\overline{\lambda}_{\infty}}(S))$$

Our equation $T_{\overline{\lambda}_{\infty}}(\nu)=\overline{\lambda}_{\infty}$ becomes then:

 $$\frac{\partial\overline{\lambda}_{\infty}}{\partial\nu}=\alpha_{\nu,\overline{\lambda}_{\infty}}E_{\overline{\Pi}_{\nu}} [H'(LS-\Psi_{\nu,\overline{\lambda}_{\infty}}(S))|S]=\alpha_{\nu,\overline{\lambda}_{\infty}}.H'(\Psi_{\nu,\overline{\lambda}_{\infty}}^{\prime}(S)S-\Psi_{\nu,\overline{\lambda}_{\infty}}(S))$$

Finally $\frac{\partial\overline{\lambda}_{\infty}}{\partial\nu}$ is also the quotient  $\frac{\mathcal{N}}{f_{\nu}}$ of the densities with respect to the Lebesgue measure. We get therefore for almost every $s$:

$$ 
\mathcal{N}(s)=\alpha_{\nu,\overline{\lambda}_{\infty}}.H'(\Psi_{\nu,\overline{\lambda}_{\infty}}^{\prime}(s)s-\Psi_{\nu,\overline{\lambda}_{\infty}}(s))f_{\nu}(s)
$$

Combining this equation with equation (\ref{eq:eq_diff_density}), we get for almost every $s$:

$$ 
\mathcal{N}(s)=\alpha_{\nu,\overline{\lambda}_{\infty}}.H'(\Psi_{\nu,\overline{\lambda}_{\infty}}^{\prime}(s)s-\Psi_{\nu,\overline{\lambda}_{\infty}}(s))f_{\mu}(\gamma_{\nu}(s))g(s)
$$

From this equation we get $g(s)=\frac{\alpha_{\nu,\overline{\lambda}_{\infty}}.H'(\Psi_{\nu,\overline{\lambda}_{\infty}}^{\prime}(s)s-\Psi_{\nu,\overline{\lambda}_{\infty}}(s))f_{\mu}(\gamma_{\nu}(s))}{\mathcal{N}(s)}$ almost surely.

Since the right hand side of this equality is continuous, it is the continuous version of $g$ we were seeking. $\Psi_{\nu,\overline{\lambda}_{\infty}}$ is thus $C^2$ which is the first claim of our proposition and for every $s$: 

$$\Psi_{\nu,\overline{\lambda}_{\infty}}^{\prime\prime}(s)=\frac{\alpha_{\nu,\overline{\lambda}_{\infty}}.H'(\Psi_{\nu,\overline{\lambda}_{\infty}}^{\prime}(s)s-\Psi_{\nu,\overline{\lambda}_{\infty}}(s))f_{\mu}( \Psi_{\nu,\overline{\lambda}_{\infty}}^{\prime}(s))}{\mathcal{N}(s)}$$

The pair $(\psi,c):=(\Psi_{\nu,\overline{\lambda}_{\infty}},\frac{1}{\alpha_{\nu,\overline{\lambda}_{\infty}}})$ is then a solution of the first equation of the differential problem $\mathcal{D}$.

It also satisfies the following ones. Indeed, since $[0,1]$ is the support of $\mu$ we get:

\[
\begin{cases}
lim_{s \rightarrow -\infty} \Psi_{\nu,\overline{\lambda}_{\infty}}^{\prime}(s)=lim_{s \rightarrow -\infty} F_{\mu}^{-1}(F_{\nu}(s))= 0\\
lim_{s \rightarrow +\infty} \Psi_{\nu,\overline{\lambda}_{\infty}}^{\prime}(s)=lim_{s \rightarrow + \infty} F_{\nu}^{-1}(F_{\mu}(s))=1.
\end{cases}
\]

Furthermore we have $ \Psi_{\nu,\overline{\lambda}_{\infty}}=\Gamma_{\nu}-E_{\overline{\lambda}_{\infty}}[\Gamma_{\nu}]$ and thus:

\[\int_{-\infty}^{+\infty}  \Psi_{\nu,\overline{\lambda}_{\infty}}(z)\mathcal{N}(z)dz= E_{\overline{\lambda}_{\infty}}[\Psi_{\nu,\overline{\lambda}_{\infty}}] =0\]
\end{proof}

\subsection{The unique solution of the differential system $\mathcal{D}$.}\label{annexe:unique_solution}

The proof of Theorem \ref{Thm_uniqueness} is made of several lemma that are presented below.

Let $(\psi_1,c_1)$ and $(\psi_2,c_2)$ be two solutions of the problem $\mathcal{D}$. Without loss of generality we may assume that $c_1 \geq c_2$. Furthermore, we have $c_2>0$. Indeed $\mathcal{D}$-2 and $\mathcal{D}$-3 imply that $\psi_i''$ must be strictly positive at some point $s$. At that point $\mathcal{D}$-1 indicates that $c_2>0$ since both $f_{\mu}$, $H'$ and $\mathcal{N}$ are positive.\\

For all this section, we set $\theta:= \psi_1-\psi_2$. $\theta$ is then a $C^2$ function. Indeed, $\psi_i$ solves  which already implies that $\psi_i$ has a second order derivative and is therefore $C^1$. Since we may express the second order derivative from equation $\mathcal{D}$-1 as a continuous function of $s$,$\psi(s)$ and $\psi'(s)$, we have that $\psi''(s)$ is itself continuous and $\psi(s)$ is $C^2$.
\\

According to $\mathcal{D}$-4 we have $\int_{-\infty}^{+\infty}  \theta(z)\mathcal{N}(z)dz=0$. Since $\theta$ is a continuous function and $\mathcal{N}(z)>0$ for all $z$, there exists $s_0$ such that $\theta(s_0)=0$.\\

Let us define:
 $$\Gamma^{+}:=\{s > s_0 \mid \theta'(s)=0 \}$$

and:

$$\Gamma^{-}:=\{s < s_0 \mid \theta'(s)=0 \}.$$ 
We also define $s^{+}:= \inf \Gamma^{+}$ and $s^{-}:= \sup \Gamma^{-}$

Observe that $\theta'(s)$ can not vanish on $]s_0,s^+[$ nor on $=]s^-,s_0[$. $\theta'$ has thus a constant sign on each interval.

\begin{lemma}\label{Not_positive}

Suppose that $s_0<s^+$. Then $\theta'<0$ on $]s_0,s^+[$.
\end{lemma}

\begin{proof}
~~\\

Assume on the contrary that $s_0<s^+$ and $\theta'>0$ on $]s_0,s^+[$.\\

\underline{First case:} suppose that $s^+<+\infty$.

Since $\theta$ is increasing on $]s_0,s^+[$:

\begin{align}\label{eq1}
\theta(s^+)>\theta(s_0)=0
\end{align}

and also, by definition of $s^+$, 

\begin{align}\label{eq2} \theta'(s^+)=0, \text{ so } \alpha:= \psi'_1(s^+)=\psi'_2(s^+) 
\end{align}

From $\mathcal{D}$-1 we get:

\begin{align*}
&\theta''(s^+)&=& \frac{\mathcal{N}(s^+)}{f_{\mu}(\alpha)}(\frac{c_1}{H'(s^+\alpha-\psi_1(s^+))}-\frac{c_2}{H'(s^+\alpha-\psi_2(s^+))})\\
&&\geq& \frac{\mathcal{N}(s^+)c_2}{f_{\mu}(\alpha)}(\frac{1}{H'(s^+\alpha-\psi_1(s^+))}-\frac{1}{H'(s^+\alpha-\psi_2(s^+))})>0\\
\end{align*}
since $\psi_1(s^+)>\psi_2(s^+)$, as indicates equation (\ref{eq1}) and since $H'$ is strictly increasing and $c_2>0$. But this is not possible since $\theta'(s^+)=0$ and $\theta'>0$ on $]s_0,s^+[$ (which implies that $\theta''(s^+) \leq 0$).

~~\\

\underline{Second case:} suppose now that $s^+ = +\infty$. It is convenient in this case to introduce the function $R$ on $[0,1]$:  $R(u)=\theta'(F_{\mathcal{N}}^{-1}(u))$ where $F_{\mathcal{N}}$ is the cumulative function of the normal law $\mathcal{N}(0,1)$. We first prove that $lim_{u \rightarrow 1} R'(u)>0$. Indeed:

\begin{equation}\label{eq:R'}
\begin{array}{rl}
\displaystyle lim_{u \rightarrow 1} R'(u)=&\displaystyle lim_{s \rightarrow +\infty} \frac{\theta''(s)}{\mathcal{N}(s)}\\
\displaystyle=& \displaystyle lim_{s \rightarrow +\infty}\frac{c_1}{f_{\mu}(1)H'(\Lambda_1(s))}-\frac{c_2}{f_{\mu}(1)H'(\Lambda_2(s))}\\
\displaystyle \geq& \displaystyle lim_{s \rightarrow +\infty}\frac{c_2}{f_{\mu}(1)} (\frac{1}{H'(\Lambda_1(s))}-\frac{1}{H'(\Lambda_2(s))})    
\end{array}  
\end{equation}
where $\Lambda_i:= s\psi'_i(s)-\psi_i(s)$. 

Now we claim that both $\Lambda_i(s)$ have a finite limit when $s \rightarrow \infty$ and that 

\begin{equation}\label{eq:ineq-limit}
lim_{+\infty} \Lambda_1(s) < lim_{+\infty} \Lambda_2(s).
\end{equation}

To see that $\Lambda_i$ has a limit as $s\rightarrow \infty$,  just observe that $\Lambda_i'(s)=s\psi_i''(s)$ which is positive for positive $s$ since $\psi_i$ is solution of $\mathcal{D}$-1 and thus $\psi_i''(s)>0$. As an increasing function, $\Lambda_i$ has therefore a limit as $s \rightarrow\infty$.

Furthermore, from assumptions \textbf{A1} and \textbf{A2}, we know that $ f_{\mu}$ and $H'$ are valued in a bounded interval $[K_1,K_2]$ with $K_1>0$. As a consequence, $\frac{c_i}{f_{\mu}(s)H'(\Lambda_i(s))}$ is bounded by a constant $K$, and therefore:

$$\Lambda_i(s)=\Lambda_i(0)+\int_0^s \Lambda_i'(u)du = -\psi_i(0)+\int_0^s u\psi_i''(u)du \leq -\psi_i(0)+K \int_0^s u \mathcal{N}(u)du$$
where the last integral converges as $s$ goes to $+\infty$.\\

We next prove inequality (\ref{eq:ineq-limit}). First observe that since $\psi_1$ and $\psi_2$ satisfy $\mathcal{D}$-2 we get  $ lim_{s \rightarrow +\infty} \theta'(s)=0$. Therefore: $$| \theta'(s)|=|\int_s^{\infty} \theta''(u)du |= |\int_s^{\infty} \psi_2''(s)-\psi_1''(s) ds| \leq 2 K \int_s^{\infty} \mathcal{N}(u)du $$ and thus $\lim_{s \rightarrow +\infty} s \theta'(s) =0$. We then conclude that: 

$$\lim_{s \rightarrow +\infty} \Lambda_1(s)-\Lambda_2(s) = \lim_{s \rightarrow +\infty}s\theta'(s)-\theta(s)=-\lim_{s \rightarrow +\infty}\theta(s).$$


Next observe that $\lim_{s \rightarrow +\infty}\theta(s)>0$. Indeed: $\theta(s_0)=0$ and $\forall s \in [s_0,+\infty[$, we have $\theta'(s)>0$. Therefore, as announced, $\lim_{s \rightarrow +\infty} \Lambda_1(s) < \lim_{s \rightarrow +\infty} \Lambda_2(s)$. This implies with equation (\ref{eq:R'}) that:

\begin{equation}\label{eq:eq1}
lim_{u \rightarrow 1} R'(u)>0
\end{equation}

Note that, according to the definition of $R$ and the fact that  $\theta'>0$ on $ ]s_0,+\infty[ $, we get:

\begin{equation}\label{eq:eq2}
R(x)> 0  \text{ for } x \in ]F_{\mathcal{N}}(s_0),1[
\end{equation} 
Finally,
\begin{equation}\label{eq:eq3}
lim_{u \rightarrow 1}R(u)=lim_{u \rightarrow 1}\theta'(F_{\mathcal{N}}^{-1}(u))=lim_{u \rightarrow +\infty}\theta'(u)=0
\end{equation}
but relations (\ref{eq:eq2}) and (\ref{eq:eq3}) are in contradiction with (\ref{eq:eq1}). This conclude the proof of the lemma.
\end{proof}

A similar argument leads to a dual result on the left side of $s_0$:

\begin{lemma}\label{lemma:dual_result}
Suppose that $s^-<s_0$. Then $\theta'>0$ on $]s^-,s_0[$.
\end{lemma}

\begin{lemma}\label{all_equal_zero}
$\theta(s_0)=\theta'(s_0)=\theta''(s_0)=0$.
\end{lemma}

\begin{proof}
Suppose $\theta'(s_0)>0$. There must exist $\delta>0$ such that $\theta'(s)>0$ for s $\in ]s_0,s_0+\delta[$. The definition of $s^+$ implies therefore $s^+ \geq s_0 + \delta > s_0$. Furthermore, $\theta'$ is strictly positive on $]s_0,s^+[$. But this is in contradiction with Lemma \ref{Not_positive}. Similarly, the assumption $\theta'(s_0)<0$ is in contradiction with the dual result Lemma \ref{lemma:dual_result}. And we must therefore have $\theta'(s_0)=0$.

Suppose now that $\theta''(s_0)>0$. Then there exists $\epsilon>0$ such that $\theta'>0$ on $]s_0,s_0+\epsilon[$ in contradiction with Lemma \ref{Not_positive}. With the same arguments, it is impossible that $\theta''(s_0)<0$ and the lemma is proved.
\end{proof}

\begin{lemma}
$c_1=c_2$.
\end{lemma}

\begin{proof}
Indeed, equation $\mathcal{D}$-1 gives, for $i=1,2$:

\[c_i= \frac{f_{\nu}(\psi'_i(s_0))\psi''_i(s_0)H'(s_0\psi'_i(s_0)-\psi_i(s_0))}{\mathcal{N}(s_0)}
\]

But, according to Lemma \ref{all_equal_zero} the right hand side does not depend on $i$.
\end{proof}

\begin{proof}[\underline{Proof of Theorem \ref{Thm_uniqueness}}]~~\\


Let $c$ denote the common value $c:=c_1=c_2$. Our two functions $\psi_1$ and $\psi_2$ are now solutions to the same differential equation:

\[\psi''(s)=F(s,\psi(s),\psi'(s))\]

where 
\[F(s,x,y):=\frac{c\mathcal{N}(s)}{  H'(sy-x)       f_{\mu}(y) }\]

Due to our assumptions \textbf{A1}, \textbf{A2} on $f_{\mu}$ and $H$, $F$ is $C^1$ with respect to $(s,x,y)$. Therefore, according to Cauchy-Lipschitz theorem, $\psi_1$ and $\psi_2$ must coincide since they are both solution of the same differential equation and have the same initial conditions $\psi(s_0),\psi'(s_0)$ at $s=s_0$.
\end{proof}

\bibliographystyle{plainnat}
\bibliography{bibriskaversion}
\end{document}